\documentclass[a4paper, 11pt]{amsart}
\usepackage{amsmath, amssymb, amsfonts, amsthm, enumerate, mathtools, tikz-cd, hyperref, stmaryrd}
\usepackage[margin=1.2in]{geometry}

\newtheorem{thm}{Theorem}[section]
\newtheorem{lem}[thm]{Lemma}
\newtheorem{defi}[thm]{Definition}
\newtheorem{rem}[thm]{Remark}

\newtheorem{prop}[thm]{Proposition}
\newtheorem{cor}[thm]{Corollary}

\newtheorem{thmA}{Theorem}

\newtheorem{corB}{Corollary}

\newtheorem{setup}[thm]{Setup}

\def\QQ{\mathbb{Q}}
\def\ZZ{\mathbb{Z}}
\def\FF{\mathbb{F}}
\def\TT{\mathbb{T}}
\def\frob{\text{Frob}}

\def\fm{\mathfrak{m}}
\def\ord{\text{ord}}
\def\st{\text{st}}

\def\eis{\text{eis}}
\def\gal{\text{Gal}}
\newcommand{\GL}{\mathrm{GL}}

\def\Ga1{\Gamma_1}
\def\rhob{\bar\rho}

\def\tr{\operatorname{tr}}
\def\red{\text{red}}
\def\univ{\text{univ}}

\def \Hom{\text{Hom}}
\def\tan{\text{tan}}
\def\defo{\text{def}}
\def\ps{\text{pd}}
\DeclareMathOperator{\rank}{rank}

\DeclareFontFamily{U}{wncy}{}
\DeclareFontShape{U}{wncy}{m}{n}{<->wncyr10}{}
 \DeclareSymbolFont{mcy}{U}{wncy}{m}{n}
 \DeclareMathSymbol{\Sh}{\mathord}{mcy}{"58} 

\begin{document}
\baselineskip 17.5pt

\title{The Eisenstein ideal of weight $k$ and ranks of Hecke algebras}
\author{Shaunak V. Deo} 
\email{shaunakdeo@iisc.ac.in}
\address{Department of Mathematics, Indian Institute of Science, Bangalore 560012, India}
\date{}


\dedicatory{Dedicated to the memory of my father Vilas G. Deo}

\begin{abstract}
Let $p$ and $\ell$ be primes such that $p > 3$ and $p \mid \ell-1$ and $k$ be an even integer. We use deformation theory of pseudo-representations to study the completion of the Hecke algebra acting on the space of cuspidal modular forms of weight $k$ and level $\Gamma_0(\ell)$ at the maximal Eisenstein ideal containing $p$.
We give a necessary and sufficient condition for the $\ZZ_p$-rank of this Hecke algebra to be greater than $1$ in terms of vanishing of the cup products of certain global Galois cohomology classes.
We also recover some of the results proven by Wake and Wang-Erickson for $k=2$ using our methods.
In addition, we prove some $R=\TT$ theorems under certain hypothesis.
\end{abstract} 

\maketitle

\section{Introduction}
Let $p$ and $\ell$ be primes such that $p > 3$ and $p \mid \ell-1$.
Let $\TT$ be the Hecke algebra over $\ZZ_p$ acting faithfully on the space of modular forms of level $\Gamma_0(\ell)$ and weight $k$ and $\fm$ be its Eisenstein maximal ideal containing $p$ i.e. the maximal ideal of $\TT$ generated by $p$ and the prime ideal corresponding to the classical Eisenstein series of level $\Gamma_0(\ell)$ and weight $k$ having Atkin-Lehner eigenvalue $-1$.
Let $\TT_{\fm}$ be the completion of $\TT$ at $\fm$ and let $\TT^0_{\fm}$ be its cuspidal quotient.

In the setting given above, Mazur, in his landmark work on Eisenstein ideal (\cite{M2}), studied the cuspidal Hecke algebra $\TT^0_{\fm}$ in the case of $k=2$. He proved (among many other things) that $\TT^0_{\fm} \neq 0$ and also asked whether one can say anything about the $\ZZ_p$-rank of $\TT^0_{\fm}$.
Since then this question has been studied in detail by various people using different approaches. 
We will now give a brief summary of their works on the $\ZZ_p$-rank of $\TT^0_{\fm}$ when $k=2$.

\subsection{History}
In \cite{Mer}, Merel proved that the $\ZZ_p$-rank of $\TT^0_{\fm}$ is greater than $1$ if and only if the image of $\prod_{i=1}^{\frac{\ell-1}{2}}i^i$ in $(\ZZ/\ell\ZZ)^{\times}$ is a $p$-th power. His method was mainly based on computation of some Eisenstein elements in the first homology group of a modular curve.
In \cite{L2}, Lecouturier extended Merel's result by combining same circle of ideas with new methods.
In particular, he gave a necessary and sufficient condition for the $\ZZ_p$-rank of $\TT^0_{\fm}$ to be greater than $2$ in terms of a numerical invariant which is similar to the Merel's invariant mentioned above (see \cite[Theorem $1.2$]{L2}).
He also gave an alternate proof of Merel's result in \cite{L2}.

On the other hand, in \cite{CE}, Calegari and Emerton studied this question using deformation theory of Galois representations. They proved that $\TT^0_{\fm} = \ZZ_p$ if the $p$-part of the class group of $\QQ(\ell^{1/p})$ is cyclic. In \cite{WWE1}, Wake and Wang-Erickson used techniques from deformation theory of Galois pseudo-representations to prove that the $\ZZ_p$-rank of $\TT^0_{\fm}$ is greater than $1$ if and only if the cup products of certain global Galois cohomology classes vanish. They also recovered many results of Calegari-Emerton. The key step in both these works is a suitable $R=\TT$ theorem.
In \cite{WWE2}, Wake and Wang-Erickson studied this question in the case of squarefree level.
We refer the reader to the well-written introductions of \cite{Mer}, \cite{CE}, \cite{WWE1} \cite{WWE2} and \cite{L2} for a summary of the known results, nice exposition of various approaches to the problem and their comparison. 

One can say that the approach of Merel and Lecouturier is on the `analytic side' and the approach of Calegari--Emerton and Wake--Wang-Erickson is on the `algebraic side'.
In \cite{W}, Wake studied the Hecke algebras $\TT_{\fm}$ and $\TT^0_{\fm}$ and their Eisenstein ideals for weights $k > 2$ by unifying the two approaches mentioned above.
In particular, he gave a necessary and sufficient condition (in terms of the Eisenstein ideal and the derivative of Mazur--Tate $\zeta$-function that he defines) for the $\ZZ_p$-rank of $\TT^0_{\fm}$ to be $1$. This is an analogue of Merel's result (\cite[Th\'{e}oreme $2$]{Mer}) for higher weights.

\subsection{Aim and Setup}
The main aim of this article is to obtain necessary and sufficient conditions for $\rank_{\ZZ_p}(\TT^0_{\fm}) \geq 2$ when $k > 2$ in terms of vanishing of cup products of certain global Galois cohomology classes and class groups.
In particular, we prove analogues of \cite[Theorem $1.2.1$]{WWE1} and \cite[Corollary $1.2.2$]{WWE1} for $k > 2$ and we recover these results when $k=2$.

Our approach is based on deformation theory of Galois representations and pseudo-representations. So it is similar to the approach of \cite{CE} and \cite{WWE1}.
However, our methods are different. To be precise, even though our main tool is comparison between deformation rings (of either representations or pseudo-representations) and Hecke algebras, our main results are not based on $R=\TT$ theorems. We instead use the description of $\dfrac{\FF_p[\epsilon]}{(\epsilon^2)}$-valued ordinary pseudo-representations, analysis of pseudo-representations arising from actual representations and results from \cite{W} and \cite{M2}. 
The results from \cite{W} that we use are about the reducibility properties of the $\TT_{\fm}$-valued pseudo-representation (\cite[Theorem $5.1.1$]{W}) and the index of Eisenstein ideal in $\TT^0_{\fm}$ (\cite[Theorem $5.1.2$]{W}).

Note that, in \cite{WWE1}, Wake and Wang-Erickson work with pseudo-representations which are finite flat at $p$ (a notion that they define and study in \cite{WWE}).
But since this condition is not present in weight $k > 2$, we work with pseudo-representations that are ordinary at $p$ and recover the results of Wake and Wang-Erickson mentioned above using them.
In addition, we also prove some $R=\TT$ theorems in certain cases.

Note that in the case of $k=2$, we need \cite[Proposition II.$9.6$]{M2}. But it is not needed in the works of Wake--Wang-Erickson (\cite{WWE1}) and Calegari--Emerton (\cite{CE}). Moreover, both sets of authors recover \cite[Proposition II.$9.6$]{M2} using their methods.

Before stating our main results, we describe the setup with which we will be working.

\begin{setup} 
\label{basic}
Let $p > 3$ be a prime and $\ell$ be a prime such that $p \mid \ell-1$. 
Let $G_{\QQ,p\ell}$ be the Galois group of the maximal extension of $\QQ$ unramified outside $p$, $\ell$ and $\infty$ over $\QQ$ and let $G_{\QQ,p}$ be the Galois group of the maximal extension of $\QQ$ unramified outside $p$ and $\infty$ over $\QQ$. 
Denote the mod $p$ cyclotomic character of $G_{\QQ,p\ell}$ by $\omega_p$ and the $p$-adic cyclotomic character by $\chi_p$.
By abuse of notation, we will also denote the mod $p$ cyclotomic character of $G_{\QQ,p}$ by $\omega_p$.
Let $k \geq 2$ be an even integer and $\rhob_0 : G_{\QQ,p\ell} \to \GL_2(\FF_p)$ be the continuous, odd representation given by $\rhob_0 = 1 \oplus \omega_p^{k-1}$.
Let $\zeta_p$ denote a primitive $p$-th root of unity. 
Suppose the following hypotheses hold:
\begin{enumerate}
\item $p-1 \nmid k$,
\item\label{don} the $\omega_p^{1-k}$-component of the $p$-part of the class group of $\QQ(\zeta_p)$ is trivial,
\item\label{teen} $\dim_{\FF_p}(H^1(G_{\QQ,p},\omega_p^{k-1})) =1$.
\end{enumerate}
\end{setup}
For a positive integer $n$, let $B_n$ be the $n$-th Bernoulli number.

\begin{rem}
\label{rem1}
Using the Herbrand--Ribet theorem and Kummer's congruences, we conclude that Condition~\eqref{don} of Setup~\ref{basic} holds if and only if $p$ does not divide $B_k$.
Combining this with Kummer's congruences, we get that $\zeta(1-k) \in \ZZ_{(p)}^{\times}$. Hence the hypotheses of \cite{W} are satisfied in our setup.
\end{rem}

\begin{rem}
\label{rem2}
From \cite[Lemma 21]{BK}, we know that Condition~\eqref{teen} of Setup~\ref{basic} holds if and only if the $\omega_p^{p+1-k}$-component of the $p$-part of the class group of $\QQ(\zeta_p)$ is trivial.
Let $0 < k_0 <p-1$ be the integer such that $k \equiv k_0 \pmod{p-1}$.
Hence, by combining the reflection principle (\cite[Theorem 10.9]{Wash}) and the Herbrand--Ribet theorem, Condition~\eqref{teen} of Setup~\ref{basic} holds if $p \nmid B_{p+1-k_0}$.
\end{rem}

\begin{rem}
Combining Remarks~\ref{rem1} and \ref{rem2}, we get that Conditions~\eqref{don} and \eqref{teen} of Setup~\ref{basic} hold if one of the following conditions hold:
\begin{itemize}
\item $p$ is a regular prime.
\item Vandiver's conjecture holds for $p$.
\item $p \nmid B_kB_{p+1-k_0}$, where $0 < k_0 <p-1$ is the integer such that $k \equiv k_0 \pmod{p-1}$.
\item $p > 7$ and $k=4,6$.
\item $p \equiv 3 \pmod{4}$ and $k= \dfrac{p+1}{2}$.
\end{itemize}
Note that we get Condition~\eqref{teen} of Setup~\ref{basic} for $k=4$ from \cite[Corollary 3.8]{K} and for $k=6$ from \cite[Corollary 7.1]{Kup}.
On the other hand, Conditions~\eqref{don} and \eqref{teen} for $k = \dfrac{p+1}{2}$ follow from the Herbrand--Ribet Theorem and \cite[Theorem 1.1]{O}.
\end{rem}

In the rest of the article, we assume that we are in Setup~\ref{basic} unless mentioned otherwise.

Let $\TT_{\fm}$ be the Hecke algebra of level $\Gamma_0(\ell)$ and weight $k$ as defined in \S\ref{heckesec} and $\TT^0_{\fm}$ be its cuspidal quotient.

Denote the absolute Galois group of $\QQ_p$ and $\QQ_{\ell}$ by $G_{\QQ_p}$ and $G_{\QQ_{\ell}}$, respectively and denote their inertia subgroups by $I_p$ and $I_{\ell}$, respectively. 
Now our assumptions imply that $ \dim_{\FF_p}(\ker(H^1(G_{\QQ,p\ell},\omega_p^{1-k}) \to H^1(I_p,\omega_p^{1-k})))=1$ (see Lemma~\ref{cohomlem}).
Choose a generator $c_0 \in \ker(H^1(G_{\QQ,p\ell},\omega_p^{1-k}) \to H^1(I_p,\omega_p^{1-k}))$.
Let $\rhob_{c_0} : G_{\QQ,p\ell} \to \GL_2(\FF_p)$ be the representation given by $\rhob_{c_0} = \begin{pmatrix} 1 & *\\ 0 & \omega_p^{k-1} \end{pmatrix}$ where $*$ corresponds to $c_0$.

 Note that both $\ker(H^1(G_{\QQ,p\ell},1) \to H^1(I_p,1))$ and $H^1(G_{\QQ,p},\omega_p^{k-1})$ are also one dimensional.
Choose generators $a_0 \in \ker(H^1(G_{\QQ,p\ell},1) \to H^1(I_p,1))$ and $b_0 \in H^1(G_{\QQ,p},\omega_p^{k-1})$.
Denote the cup product of $c_0$ and $b_0$ by $c_0 \cup b_0$ and the cup product of $c_0$ and $a_0$ by $c_0 \cup a_0$.
So in particular, $c_0 \cup b_0 \in H^2(G_{\QQ,p\ell},1)$ and $c_0 \cup a_0 \in H^2(G_{\QQ,p\ell}, \omega_p^{1-k})$.

\subsection{Main Results}
We are now ready to state main results.
\begin{thmA}[see Corollary~\ref{cor1}, Corollary~\ref{cor2}, Theorem~\ref{thm2}]
\label{thma}
Suppose we are in Setup~\ref{basic}. Then:
\begin{enumerate}
\item\label{part:1} If $k=2$, then $\rank_{\ZZ_p}(\TT^0_{\fm}) = 1$ if and only if $c_0 \cup a_0 \neq 0$.
\item\label{part:2} If $k > 2$, then $\rank_{\ZZ_p}(\TT^0_{\fm}) = 1$ if and only if $c_0 \cup b_0 \neq 0$ and $c_0 \cup a_0 \neq 0$.
\end{enumerate}
\end{thmA}

Note that part~\eqref{part:1} of Theorem~\ref{thma} has already been proven by Wake and Wang-Erickson in \cite{WWE1} using similar approach but different methods.

In \cite{W}, Wake has proved that when $k > 2$, $\rank_{\ZZ_p}(\TT^0_{\fm}) = 1$ if and only if the Eisenstein ideal of $\TT^0_{\fm}$ is principal and a certain element $\xi'_{\text{MT}} \in \FF_p$ (that he defines in \cite[Section $1.2.2$]{W}) is non-zero. See \cite[Theorem $1.2.4$]{W} for more details.
We don't use this result to prove part~\eqref{part:2} of Theorem~\ref{thma} but we do need some other results from \cite{W}.

To be precise, when $c_0 \cup b_0 =0$, we prove part~\eqref{part:2} of Theorem~\ref{thma} by proving that the Eisenstein ideal of $\TT^0_{\fm}$ is not principal (see Theorem~\ref{thm2} and Theorem~\ref{thm3}).
As a consequence of our analysis, we get the following result regarding the principality of the Eisenstein ideal of $\TT^0_{\fm}$:
\begin{corB}
\label{main}
Suppose we are in Setup~\ref{basic} and $k >2$.
Then the Eisenstein ideal of $\TT^0_{\fm}$ is principal if and only if $c_0 \cup b_0 \neq 0$.
Moreover, if Vandiver's conjecture holds for $p$, then these assertions hold if and only if $\prod_{j=1}^{p-1}(1-\zeta_p^j)^{j^{k-2}} \in (\ZZ/\ell\ZZ)^{\times}$ (where $\zeta_p \in \ZZ/\ell\ZZ$ is a primitive $p$-th root of unity) is not a $p$-th power.
\end{corB}

\begin{rem}
Note that Corollary~\ref{main} matches with the prediction made by Wake in \cite[Section $1.2.3$, Remark 3.2.1]{W}.
\end{rem}

\begin{rem}
If $p$ is a regular prime, then Vandiver's conjecture holds for $p$.
\end{rem}

When $c_0 \cup b_0 \neq 0$, the Eisenstein ideal is principal.
In this case, we prove that $\rank_{\ZZ_p}(\TT^0_{\fm}) = 1$ if and only if $c_0 \cup a_0 \neq 0$ by using an analysis of pseudo-representations arising from representations.

Let $\zeta_{\ell}$ be a  primitive $\ell$-th root of unity and let ${\zeta_{\ell}}^{(p)} \in \QQ(\zeta_{\ell})$  be an element such that $[\QQ(\zeta_{\ell}^{(p)}) : \QQ] = p$. 
Denote by $\text{Cl}(\QQ(\zeta_{\ell}^{(p)},\zeta_p))$ the class group of $\QQ(\zeta_{\ell}^{(p)},\zeta_p)$ and let $(\text{Cl}(\QQ(\zeta_{\ell}^{(p)},\zeta_p))/\text{Cl}(\QQ(\zeta_{\ell}^{(p)},\zeta_p))^p)[\omega_p^{1-k}]$ be the subspace of $\text{Cl}(\QQ(\zeta_{\ell}^{(p)},\zeta_p))/\text{Cl}(\QQ(\zeta_{\ell}^{(p)},\zeta_p))^p$ on which $\text{Gal}(\QQ(\zeta_p)/\QQ)$ acts by $\omega_p^{1-k}$.
Now we get the following corollaries (see Proposition~\ref{classgroupprop}):

\begin{corB}
\label{corb}
Suppose we are in Setup~\ref{basic} and $k=2$. Then the following are equivalent:
\begin{enumerate}
\item $\rank_{\ZZ_p}(\TT^0_{\fm}) = 1$.
\item $\dim_{\FF_p}((\text{Cl}(\QQ(\zeta_{\ell}^{(p)},\zeta_p))/\text{Cl}(\QQ(\zeta_{\ell}^{(p)},\zeta_p))^p)[\omega_p^{1-k}]) = 1$.
\item $\prod_{i=1}^{\frac{\ell -1}{2}}i^{i}$ in $(\ZZ/\ell \ZZ)^{\times}$ is not a $p$-th power.
\end{enumerate}
\end{corB}

Note that Corollary~\ref{corb} has already been proved by Wake and Wang-Erickson in \cite{WWE1} and by Lecouturier in \cite{L2} using different methods. However, Wake and Wang-Erickson use results of \cite{L} to prove that the second part of Corollary~\ref{corb} implies the first part.
We give a slightly different proof of the same (see Proposition~\ref{classgroupprop} and its proof).

\begin{corB}
\label{corc}
Suppose we are in Setup~\ref{basic} and $k>2$. Then the following are equivalent:
\begin{enumerate}
\item $\rank_{\ZZ_p}(\TT^0_{\fm}) = 1$.
\item $\dim_{\FF_p}((\text{Cl}(\QQ(\zeta_{\ell}^{(p)},\zeta_p))/\text{Cl}(\QQ(\zeta_{\ell}^{(p)},\zeta_p))^p)[\omega_p^{1-k}]) = 1$ and the restriction map $H^1(G_{\QQ,p},\omega_p^{k-1}) \to H^1(G_{\QQ_{\ell}},\omega_p^{k-1})$ is not the zero map.
\item $\prod_{i=1}^{\ell -1}i^{(\sum_{j=1}^{i-1}j^{k-1})} \in (\ZZ/\ell \ZZ)^{\times}$ is not a $p$-th power and the restriction map $H^1(G_{\QQ,p},\omega_p^{k-1}) \to H^1(G_{\QQ_{\ell}},\omega_p^{k-1})$ is not the zero map.
\end{enumerate}
Moreover, if Vandiver's conjecture holds for $p$, then the above assertions hold if and only if  $\prod_{i=1}^{\ell -1}i^{(\sum_{j=1}^{i-1}j^{k-1})} \in (\ZZ/\ell \ZZ)^{\times}$ is not a $p$-th power and $\prod_{i=1}^{p-1}(1-\zeta_p^i)^{i^{k-2}}  \in (\ZZ/\ell\ZZ)^{\times}$ (where $\zeta_p \in \ZZ/\ell\ZZ$ is a primitive $p$-th root of unity) is not a $p$-th power.
\end{corB}

Let $R^{\ps,\ord}_{\rhob_0,k}(\ell)$ be the universal $p$-ordinary, $\ell$-unipotent deformation ring of $(\tr(\rhob_0),\det(\rhob_0))$ with determinant $\chi_p^{k-1}$ (see Definition~\ref{orddefi} and the paragraph after it).
Let $R^{\ps,\st}_{\rhob_0,k}(\ell)$ be the universal $p$-ordinary, Steinberg-or-unramified at $\ell$ deformation ring of $(\tr(\rhob_0),\det(\rhob_0))$ with determinant $\chi_p^{k-1}$ (see Definition~\ref{stdefi} and the paragraph after it).
Let $R^{\defo,\ord}_{\rhob_{c_0},k}(\ell)$ be the quotient of the universal ordinary deformation ring of $\rhob_{c_0}$ with determinant $\chi_p^{k-1}$ defined in \S\ref{defsec}.

Note that, there is a surjective map $\phi_{\TT} : R^{\ps,\ord}_{\rhob_0,k}(\ell) \to \TT_{\fm}$ such that $\phi_{\TT}$ factors through $R^{\ps,\st}_{\rhob_0,k}(\ell)$ giving the map $\psi_{\TT} :  R^{\ps,\st}_{\rhob_0,k}(\ell) \to \TT_{\fm}$ (see Lemma~\ref{pseudoordlem}).
Moreover the map $R^{\ps,\ord}_{\rhob_0,k}(\ell) \to  \TT^0_{\fm}$ obtained by composing $\phi_{\TT}$ with the surjective map $\TT_{\fm} \to \TT^0_{\fm}$ factors through $R^{\defo,\ord}_{\rhob_{c_0},k}(\ell) $ giving the map $\phi_{\TT^0} : R^{\defo,\ord}_{\rhob_{c_0},k}(\ell) \to \TT^0_{\fm}$ (see Lemma~\ref{ordlem}).

We are now ready to state the $R=\TT$ theorems that we prove. Let $(R^{\ps,\ord}_{\rhob_0,k}(\ell))^{\red}$ be the maximal reduced quotient of $R^{\ps,\ord}_{\rhob_0,k}(\ell) $.
Recall, from Corollary~\ref{main}, that $c_0 \cup b_0 \neq 0$ if and only if the Eisenstein ideal of $\TT^0_{\fm}$ is principal.
\begin{thmA}
\label{thmb}
Suppose we are in Setup~\ref{basic} and $c_0 \cup b_0 \neq 0$. Then
\begin{enumerate}
\item\label{item1} $\phi_{\TT} : R^{\ps,\ord}_{\rhob_0,k}(\ell) \to \TT_{\fm}$ induces an isomorphism $(R^{\ps,\ord}_{\rhob_0,k}(\ell))^{\red} \simeq \TT_{\fm}$ of local complete intersection rings.
\item\label{item2} $\psi_{\TT} :  R^{\ps,\st}_{\rhob_0,k}(\ell) \to \TT_{\fm}$ is an isomorphism of local complete intersection rings.
\item\label{item3} $\phi_{\TT^0} : R^{\defo,\ord}_{\rhob_{c_0},k}(\ell) \to \TT^0_{\fm}$ is an isomorphism of local complete intersection rings.
\end{enumerate}
\end{thmA}

From part~\eqref{item3} of Theorem~\ref{thmb}, we get the following analogue of \cite[Corollary $1.6$]{CE}:
\begin{corB}
Suppose we are in Setup~\ref{basic} and $c_0 \cup b_0 \neq 0$. Then the $\ZZ_p$-rank of $\TT^0_{\fm}$ is the largest integer $n$ for which there exists an ordinary deformation $\rho : G_{\QQ,p\ell} \to \GL_2(\FF_p[\epsilon]/(\epsilon^n))$ of $\rhob_{c_0}$ such that $\det(\rho) = \omega_p^{k-1}$, $\tr(\rho(g))=2$ for all $g \in I_{\ell}$ and
the set $\{\tr(\rho(g)) \mid g \in G_{\QQ,p\ell}\}$ generates $\FF_p[\epsilon]/(\epsilon^n)$ as an $\FF_p$-algebra.
\end{corB}

\subsection{Sketch of the proofs of main results}
We will now give a brief outline of the proof of Theorem~\ref{thma}.
We first analyze the space of deformations $(t,d) : G_{\QQ,p\ell} \to \FF_p[\epsilon]/(\epsilon^2)$ of $(\tr(\rhob_0),\det(\rhob_0))$ which are $p$-ordinary and $\ell$-unipotent with determinant $\chi_p^{k-1}$ to obtain its properties.
To be precise, we prove that the space of such deformations has dimension either $1$ or $2$ (see Lemma~\ref{tanlem}) and this space is one dimensional if $c_0 \cup b_0 \neq 0$ (see Lemma~\ref{cuplem}).

So we split the proof of Theorem~\ref{thma} in two cases: in the first case, we assume either $k=2$ or $c_0 \cup b_0 \neq 0$ and in the second case, we assume $k > 2$ and $c_0 \cup b_0 = 0$.
In the first case, we know that the tangent space of $\TT_{\fm}/(p)$ has dimension $1$ and hence, its Eisenstein ideal is principal.
We then prove that in this case, all the first order deformations of $(\tr(\rhob_0),\det(\rhob_0))$ arising from $\TT_{\fm}$ are reducible.
From Lemma~\ref{infinlem}, we know that these reducible pseudo-representations arise from actual representations if and only if $c_0 \cup a_0 = 0$.
On the Hecke side, we know, from Lemma~\ref{ordlem}, that the pseudo-representation $(\tau^0_{\ell},\delta^0_{\ell})  : G_{\QQ,p\ell} \to \TT^0_{\fm}$ arises from an ordinary deformation $\rho_{\TT^0} : G_{\QQ,p\ell} \to \GL_2(\TT^0_{\fm})$ of $\rhob_{c_0}$.
Therefore, combining these two facts we see that if $\rank_{\ZZ_p}(\TT^0_{\fm}) > 1$ then $c_0 \cup a_0 =0$.

On the other hand suppose $c_0 \cup a_0 =0$. Let $\phi : R^{\ps,\ord}_{\rhob_0,k}(\ell) \to R^{\defo,\ord}_{\rhob_{c_0},k}(\ell)$ be the map induced by the universal deformation taking values in $R^{\defo,\ord}_{\rhob_{c_0},k}(\ell)$ and $F : \TT_{\fm} \to \TT^0_{\fm}$ be the natural surjective map.
Then we prove, using Lemma~\ref{ordlem}, that $\phi_{\TT}(\ker(\phi)) \subset \ker(F)$ and Lemma~\ref{infinlem} implies that $\phi_{\TT}(\ker(\phi)) \subset (p,\fm^2)$.
Combining these facts along with the principality of the Eisenstein ideal, \cite[Theorem $5.1.2$]{W} and \cite[Proposition II.$9.6$]{M2} (which give the index of Eisenstein ideal in $\TT^0_{\fm}$), we show that $\rank_{\ZZ_p}(\TT^0_{\fm}) > 1$ which proves Theorem~\ref{thma} in the first case.

In the second case, we split the proof of Theorem~\ref{thma} in two steps. In the first step we prove that $\rank_{\ZZ_p}(\TT^0_{\fm}) > 1$ if $c_0 \cup b_0 =0$ and $p \mid k$ (Theorem~\ref{thm2}).
To prove this, we use the relation between the tame inertia group and the Frobenius, techniques from Generalized Matrix Algebras (GMA) along with \cite[Theorem $5.1.2$]{W} and \cite[Theorem $5.1.1$]{W} (which describes the biggest quotient of $\TT_{\fm}$ in which $(\tau_{\ell},\delta_{\ell})$ is reducible) to prove that the Eisenstein ideal is not principal.
To prove $\rank_{\ZZ_p}(\TT^0_{\fm}) > 1$ when $c_0 \cup b_0 =0$ and $p \nmid k$, we combine Theorem~\ref{thm2} and a result of Jochnowitz (\cite{J}) about finiteness of the space of $p$-ordinary modular forms modulo $p$.
Indeed, the result of Jochnowitz, along with some standard duality results, implies that the $\ZZ_p$-rank of $\TT^0_{\fm}$ is same as the $\ZZ_p$-rank of the corresponding Hecke algebra of weight $k'$ for any $k' > k$ such that $k' \equiv k \pmod{p-1}$. After taking such a $k'$ with $p \mid k'$, we use Theorem~\ref{thm2} to prove the result.

\subsection{Structure of the paper}
In \S\ref{defsec}, we define various deformation rings that we will be working with throughout the article.
In \S\ref{prelimsec}, we gather several preliminary results from deformation theory which will be used crucially in the proof of main theorems.
In \S\ref{heckesec}, we define the Hecke algebras that we will be working with and gather their properties.
In \S\ref{mainsec}, we state and prove the main theorems of this article and also state and prove their corollaries.

\subsection{Notation}
\label{notsec}
In this subsection, we will develop some notation that will be used in the rest of the article.
Recall that we denoted the absolute Galois groups of $\QQ_p$ and $\QQ_{\ell}$ by $G_{\QQ_p}$ and $G_{\QQ_{\ell}}$, respectively and their inertia groups by $I_p$ and $I_{\ell}$, respectively.
Denote the Frobenius at $\ell$ by $\frob_{\ell}$.
Fix embeddings $i_{\ell} : G_{\QQ_{\ell}} \to G_{\QQ,p\ell}$ and $i_{p} : G_{\QQ_{p}} \to G_{\QQ,p\ell}$. Note that such embeddings are well defined upto conjugacy.
For a representation $\rho$ of $G_{\QQ,p\ell}$, we denote the representation $\rho \circ i_{\ell}$ (resp. $\rho \circ i_{p}$)  by $\rho|_{G_{\QQ_{\ell}}}$ (resp. by $\rho|_{G_{\QQ_{p}}}$) and denote the restriction of $\rho|_{G_{\QQ_{\ell}}}$ (resp. of $\rho|_{G_{\QQ_{p}}}$) to $I_{\ell}$ (resp. to $I_p$) by $\rho|_{I_{\ell}}$ (resp. by $\rho|_{I_p}$).
By abuse of notation, we also denote $\omega_p|_{G_{\QQ_{p}}}$ and $\chi_p|_{G_{\QQ_{p}}}$ by $\omega_p$ and $\chi_p$, respectively.

Now $(\tr(\rhob_0), \det(\rhob_0)) : G_{\QQ,p\ell} \to \FF_p$ is a $2$-dimensional pseudo-representation of $G_{\QQ,p\ell}$ (in the sense of Chenevier (\cite{C})). See \cite[Section $1.4$]{BK} for definition and properties of $2$-dimensional pseudo-representations. 
In this article, we will only consider $2$-dimensional pseudo-representations.
If $(t,d) : G \to R$ is a pseudo-representation and $I$ is an ideal of $R$, then we denote by $(t \pmod{I} , d \pmod{I})$ the pseudo-representation $G \to R/I$ obtained by composing $(t,d)$ with the quotient map $R \to R/I$.
All the representations, pseudo-representations and cohomology groups that we consider are assumed to be continuous unless mentioned otherwise.

If $(t,d) : G_{\QQ,p\ell} \to \FF_p[\epsilon]/(\epsilon^2)$ is a pseudo-representation deforming $(\tr(\rhob_0),\det(\rhob_0))$, then we call it a first order deformation of $(\tr(\rhob_0),\det(\rhob_0))$.
If $\rho : G_{\QQ,p\ell} \to \GL_2(\FF_p[\epsilon]/(\epsilon^2))$ is a deformation of $\rhob_{c_0}$, then we call it a first order deformation of $\rhob_{c_0}$.

If $G$ is either a quotient of a class group of exponent $p$ or a Galois cohomology group of a representation of either a local or a global Galois group, then we denote by $\dim(G)$ the $\FF_p$-dimension of $G$.
If $\rho$ is a representation of $G_{\QQ,p\ell}$ and $c \in H^i(G_{\QQ,p\ell},\rho)$, then we denote by $c|_{G_{\QQ_\ell}}$ the image of $c$ under the restriction map $H^i(G_{\QQ,p\ell},\rho) \to H^1(G_{\QQ_{\ell}},\rho)$. If $c$ and $c'$ are two Galois cohomology classes (either local or global), then denote by $c \cup c'$ their cup product.

Let $\mathcal{C}$ be the category of local complete noetherian rings with residue field $\FF_p$. If $R$ is an object of $\mathcal{C}$, then denote its maximal ideal by $m_R$, denote its tangent space by $\tan(R)$ and denote the $\FF_p$ dimension of $\tan(R)$ by $\dim(\tan(R))$.
By abuse of notation, we denote the character $G_{\QQ,p\ell} \to R^{\times}$ obtained by composing $\chi_p$ with the natural map $\ZZ_p^{\times} \to R^{\times}$ either by $\chi_p$ if $p \neq 0$ in $R$ or by $\omega_p$ if $p=0$ in $R$.

Let $\nu$ be the highest power of $p$ dividing $\ell-1$ and $v_p(k)$ be the highest power of $p$ dividing $k$ (i.e. the $p$-valuation of $k$).

\subsection{Acknowledgments} 
I would like to thank Preston Wake, Frank Calegari and Carl Wang-Erickson for providing some helpful comments on an earlier draft of this article.
I would like to thank Mahesh Kakde and Bharathwaj Palvannan for some helpful discussions about the numerical criterion of Corollary~\ref{main}.
I would also like to thank the referee for a careful reading of the article and for providing numerous comments which helped immensely in improving the exposition.

\section{Deformation rings}
\label{defsec}
Let $R^{\ps}_{\rhob_0}$ be the universal deformation ring of the pseudo-representation $(\tr(\rhob_0), \det(\rhob_0)) : G_{\QQ,p\ell} \to \FF_p$ in $\mathcal{C}$. Note that the existence of $R^{\ps}_{\rhob_0}$ is proved in \cite{C}.
Let $(T^{\univ}, D^{\univ}) : G_{\QQ,p\ell} \to R^{\ps}_{\rhob_0}$ be the universal pseudo-representation deforming $(\tr(\rhob_0),\det(\rhob_0))$.
We will now define the deformation problems (and their deformation rings) that we will be working with.

\begin{defi}
\label{orddefi}
Given an object $R$ of $\mathcal{C}$, a pseudo-representation $(t,d) : G_{\QQ,p\ell} \to R$ is called a $p$-ordinary, $\ell$-unipotent deformation of $(\tr(\rhob_0), \det(\rhob_0))$ with determinant $\chi_p^{k-1}$ if the following conditions hold:
\begin{enumerate}
\item $(t,d) \pmod{m_R} = (\tr(\rhob_0),\det(\rhob_0))$,
\item $d : G_{\QQ,p\ell} \to R^{\times}$ is the character $\chi_{p}^{k-1}$,
\item $t(g)=2$ for all $g \in I_{\ell}$, 
\item\label{item5} For all $g' \in G_{\QQ,p\ell}$ and $g, h \in I_p$,
$t(g'(g-\chi_{p}^{k-1}(g))(h-1)) = 0.$
\end{enumerate}
\end{defi}

Let $R^{\ps,\ord}_{\rhob_0,k}(\ell)$ be the object of $\mathcal{C}$ representing the functor from $\mathcal{C}$ to the category of sets which sends an object $R$ of $\mathcal{C}$ to the set of $p$-ordinary, $\ell$-unipotent pseudo-representations $(t,d) : G_{\QQ,p\ell} \to R$ with determinant $\chi_p^{k-1}$ deforming $(\tr(\rhob_0),\det(\rhob_0))$.

It is easy to verify that $R^{\ps,\ord}_{\rhob_0,k}(\ell)$ exists and it is given by the quotient of $R^{\ps}_{\rhob_0}$ by the ideal $I_k$ generated by the set 
\begin{multline}
\{D^{\univ}(g) - \chi_{p}^{k-1}(g) \mid g \in G_{\QQ,p\ell}\} \cup \{T^{\univ}(g)-2 \mid g \in I_{\ell}\} \\ \cup \{T^{\univ}(g'(g-\chi_{p}^{k-1}(g))(h-1)) \mid g' \in G_{\QQ,p\ell}, g,h \in I_p\}.
\end{multline}

Note that our notion of $p$-ordinariness, given by point~\eqref{item5} of Definition~\ref{orddefi}, is inspired from the notion of ordinary pseudo-representations defined by Calegari and Specter (\cite[Definition $2.5$]{CS}). But we have slightly changed their notion to make it suitable for our purpose.

\begin{rem}
The auxiliary parameter $\alpha$ appearing in the definition of the $p$-ordinary pseudo-representations in \cite[Definition $2.5$]{CS} is required to account for the presence of the Hecke operator $T_p$ in the Hecke algebra, especially in the non-$p$-distinguished case.
But we are assuming that $k$ is even which means that $1 \neq \omega_p^{k-1}$ i.e. we are in the $p$-distinguished case. So we do not need this auxiliary parameter.
\end{rem}

\begin{defi}
\label{stdefi}
Given an object $R$ of $\mathcal{C}$, a $p$-ordinary, $\ell$-unipotent deformation $(t,d) : G_{\QQ,p\ell} \to R$ of $(\tr(\rhob_0), \det(\rhob_0))$ with determinant $\chi_p^{k-1}$ is called Steinberg-or-unramified at $\ell$ if for any lift $g_{\ell} \in G_{\QQ_{\ell}}$ of $\frob_{\ell}$, we have $$t(g(g_{\ell}-\ell^{\frac{k}{2}})(h-1))=0,$$ for every $h \in I_{\ell}$ and $g \in G_{\QQ,p\ell}$.
\end{defi}

Let $R^{\ps,\st}_{\rhob_0,k}(\ell)$ be the object of $\mathcal{C}$ representing the functor from $\mathcal{C}$ to the category of sets which sends an object $R$ of $\mathcal{C}$ to the set of $p$-ordinary, Steinberg-or-unramified at $\ell$ pseudo-representations $(t,d) : G_{\QQ,p\ell} \to R$ with determinant $\chi_p^{k-1}$ deforming $(\tr(\rhob_0),\det(\rhob_0))$.
It is easy to verify that $R^{\ps,\st}_{\rhob_0,k}(\ell)$ exists and it is given by the quotient of $R^{\ps}_{\rhob_0}$ by the ideal $J_k$ generated by the ideal $I_k$ along with the set $$\{T^{\univ}(g(g_{\ell}-\ell^{\frac{k}{2}})(h-1)) \mid g_{\ell} \in G_{\QQ_{\ell}} \text{ is a lift of } \frob_{\ell}, h \in I_{\ell}, g \in G_{\QQ,p\ell}\}.$$
Note that our notion of a Steinberg-or-unramified at $\ell$ pseudo-representation is inspired from the Steinberg at $\ell$ condition defined and studied by Wake and Wang-Erickson in \cite[Section $1.9.1$]{WWE2}.
In an unpublished version of \cite{CS}, Calegari and Specter also define a similar notion (which they call ordinary at $\ell$ pseudo-representation).

\begin{lem}
\label{cohomlem}
Suppose $p-1 \nmid k$, $k$ is even and the $\omega_p^{1-k}$-component of the $p$-part of the class group of $\QQ(\zeta_p)$ is trivial. Then $\dim(H^1(G_{\QQ,p\ell},\omega_p^{1-k}))=2$ and $\dim(\ker(H^1(G_{\QQ,p\ell},\omega_p^{1-k}) \to H^1(I_p,\omega_p^{1-k})))=1$.
\end{lem}
\begin{proof}
%
As we are assuming that the $\omega_p^{1-k}$-component of the $p$-part of the class group of $\QQ(\zeta_p)$ is trivial, it follows that $\ker(H^1(G_{\QQ,p},\omega_p^{1-k}) \to H^1(G_{\QQ_p},\omega_p^{1-k}))$ is trivial.
As $p-1 \nmid k$, it follows, from local Euler characteristic formula, that $\dim(H^1(G_{\QQ_p},\omega_p^{1-k}))=1$ and hence, $\dim(H^1(G_{\QQ,p},\omega_p^{1-k})) \leq 1$.
As $\omega_p^{1-k}$ is odd, global Euler characteristic formula implies that $\dim(H^1(G_{\QQ,p},\omega_p^{1-k})) \geq 1$. So, we have $\dim(H^1(G_{\QQ,p},\omega_p^{1-k})) = 1$.

Thus, by Greenberg-Wiles formula (\cite[Theorem $2$]{Wa}), we get that $\ker(H^1(G_{\QQ,p},\omega_p^k) \to H^1(G_{\QQ_p},\omega_p^k))$ is trivial.
Therefore, we conclude that $$\ker(H^1(G_{\QQ,p\ell},\omega_p^k) \to H^1(G_{\QQ_p},\omega_p^k) \times H^1(G_{\QQ_{\ell}},\omega_p^k)) = 0.$$
Hence, we get that 
$$\dim(H^1(G_{\QQ,p\ell},\omega_p^{1-k})) = 1+\dim(H^0(G_{\QQ_p},\omega_p^k))+\dim(H^0(G_{\QQ_{\ell}},\omega_p^k))=1+0+1=2.$$

So $1 \leq \dim(\ker(H^1(G_{\QQ,p\ell},\omega_p^{1-k}) \to H^1(G_{\QQ_p},\omega_p^{1-k}))) \leq 2$. Now we can view $H^1(G_{\QQ,p},\omega_p^{1-k})$ as a subgroup of $H^1(G_{\QQ,p\ell},\omega_p^{1-k})$ and we have seen that $\ker(H^1(G_{\QQ,p},\omega_p^{1-k}) \to H^1(G_{\QQ_p},\omega_p^{1-k}))$ is trivial.
Hence, it follows that $\dim(\ker(H^1(G_{\QQ,p\ell},\omega_p^{1-k}) \to H^1(G_{\QQ_p},\omega_p^{1-k})))=1$.
As $\omega_p^{1-k}|_{G_{\QQ_p}} \neq 1$, we see that $$\ker(H^1(G_{\QQ,p\ell},\omega_p^{1-k}) \to H^1(G_{\QQ_p},\omega_p^{1-k})) = \ker(H^1(G_{\QQ,p\ell},\omega_p^{1-k}) \to H^1(I_p,\omega_p^{1-k})).$$
This proves the lemma.
\end{proof}

Recall that we have fixed a generator $c_0 \in \ker(H^1(G_{\QQ,p\ell},\omega_p^{1-k}) \to H^1(I_p,\omega_p^{1-k}))$.
Note that there exists a $g_0 \in I_p$ such that $\omega_p^{k-1}(g_0) \neq 1$.  Fix such a $g_0 \in I_p$.
 Let $\rhob_{c_0} : G_{\QQ,p\ell} \to \GL_2(\FF_p)$ be the representation such that $\rhob_{c_0}(g_0) = \begin{pmatrix} 1 & 0\\ 0 & \omega_p^{k-1}(g_0) \end{pmatrix}$ and $\rhob_{c_0}(g) = \begin{pmatrix} 1 & *\\ 0 & \omega_p^{k-1}(g) \end{pmatrix}$ for all $g \in G_{\QQ,p\ell}$, where $*$ corresponds to $c_0$.

\begin{defi}
\label{ordinarydef}
Let $R^{\defo,\ord}_{\rhob_{c_0},k}$ be the universal ordinary deformation ring (in the sense of Mazur) of $\rhob_{c_0}$ with constant determinant $\chi_p^{k-1}$ in $\mathcal{C}$.
So it represents the functor from $\mathcal{C}$ to the category of sets which sends an object $R$ of $\mathcal{C}$ to the set of equivalence classes of representations $\rho : G_{\QQ,p\ell} \to \GL_2(R)$ such that 
\begin{enumerate}
\item $\rho \pmod{m_R} = \rhob_{c_0}$, 
\item There exists an isomorphism $\rho|_{G_{\QQ_p}} \simeq \begin{pmatrix} \eta_1 & *\\ 0 & \eta_2\end{pmatrix},$ where $\eta_2$ is an unramified character of $G_{\QQ_p}$ lifting the trivial character $1$,
\item $\det(\rho)=\chi_{p}^{k-1}$.
\end{enumerate}
\end{defi}
As $c_0 \neq 0$, the existence of $R^{\defo,\ord}_{\rhob_{c_0},k}$ follows from \cite{M} and \cite{Ra}.

Let $\rho^{\univ} : G_{\QQ,p\ell} \to \GL_2(R^{\defo,\ord}_{\rhob_{c_0},k})$ be the universal ordinary deformation of $\rhob_{c_0}$ with constant determinant $\chi_p^{k-1}$. 
Let $R^{\defo,\ord}_{\rhob_{c_0},k}(\ell) := R^{\defo,\ord}_{\rhob_{c_0},k}/I$, where $I$ is the ideal of $R^{\defo,\ord}_{\rhob_{c_0},k}$ generated by the set $\{\tr(\rho^{univ}(g)) - 2 \mid g \in I_{\ell}\}$.
Let $\rho^{\univ,\ell} : G_{\QQ,p\ell} \to \GL_2(R^{\defo,\ord}_{\rhob_{c_0},k}(\ell))$ be the representation obtained by composing $\rho^{\univ}$ with the natural surjective map $R^{\defo,\ord}_{\rhob_{c_0},k} \to R^{\defo,\ord}_{\rhob_{c_0},k}(\ell)$.

\section{Preliminary results}
\label{prelimsec}
In this section, we gather various preliminary results which will be crucially used in the proofs of the main theorems.
We begin by recalling the notion of Generalized Matrix Algebras (GMAs) from \cite{BC}.

\subsection{Generalized Matrix Algebras}
Let $R$ be a complete noetherian ring with residue field $\FF_p$ and let $A$ be a topological Generalized Matrix Algebra (GMA) over $R$ of type $(1,1)$ as defined in \cite[Section $2.2$, Section $2.3$]{Bel}. 
This means that there exist topological $R$-modules $B$ and $C$ such that $A=\begin{pmatrix} R & B\\ C & R\end{pmatrix}$ i.e. every element of $A$ can be written as $\begin{pmatrix} a & b\\c & d\end{pmatrix}$ with $a,d \in R$, $b \in B$ and $c \in C$ and there exists a continuous morphism $m : B \otimes_R C \to R$ of $R$-modules such that $A$ becomes a (not necessarily commutative) topological $R$-algebra under the addition and multiplication given by: 
\begin{enumerate}
\item Addition: $$\begin{pmatrix} a_1 & b_1\\ c_1 & d_1\end{pmatrix} + \begin{pmatrix} a_2 & b_2\\ c_2 & d_2\end{pmatrix} = \begin{pmatrix} a_1+a_2 & b_1+b_2\\ c_1+c_2 & d_1+d_2\end{pmatrix},$$
\item Multiplication: $$\begin{pmatrix} a_1 & b_1\\ c_1 & d_1\end{pmatrix} . \begin{pmatrix} a_2 & b_2\\ c_2 & d_2\end{pmatrix} = \begin{pmatrix} a_1a_2+m(b_1 \otimes c_2) & a_1b_2+d_2b_1\\ d_1c_2+a_2c_1 & d_1d_2+m(b_2 \otimes c_1)\end{pmatrix}.$$
\end{enumerate}
We refer the reader to \cite[Section $2.2$, Section $2.3$]{Bel} for more details.

From now on, we assume that all the GMAs that we consider are topological GMAs of type $(1,1)$ unless mentioned otherwise. 
\begin{defi}
\label{faithfuldef}
Let $A=\begin{pmatrix} R & B\\ C & R\end{pmatrix}$ be a GMA over $R$ as above.
Keeping the notation developed above, we say that the GMA $A$ is faithful if $m(b \otimes c)=0$ for all $c \in C$ implies that $b=0$ and $m(b \otimes c)=0$ for all $b \in B$ implies that $c=0$.
\end{defi}
Given a GMA $A$ over $R$ with $R$-modules $B$ and $C$ and the multiplication map $m : B \otimes_R C \to R$ as given above, we denote $m(b \otimes c)$ by $bc$ for all $b \in B, c \in C$ and we denote by $BC$ the image of the map $m : B \otimes_R C \to R$.

\subsection{General results}
If $R$ is a complete noetherian local ring with residue field $\FF_p$ and $(t,d) : G_{\QQ,p\ell} \to R$ is a pseudo-representation deforming $(\tr(\rhob_0),\det(\rhob_0))$, then \cite[Lemma $3.8$]{C} implies that the pseudo-representation $(t|_{I_{\ell}},d|_{I_{\ell}}) : I_{\ell} \to R$ (i.e. the restriction of the pseudo-representation $(t,d)$ to $I_{\ell}$) factors through the $\ZZ_p$-quotient of the tame inertia group at $\ell$. Fix a lift $i_{\ell} \in I_{\ell}$ of a topological generator of this $\ZZ_p$-quotient of $I_{\ell}$.

Recall that in \S\ref{defsec}, we have fixed a $g_0 \in I_p$ such that $\omega_p^{k-1}(g_0) \neq 1$.
We will now prove a result relating pseudo-representations with GMAs and establishing various properties of these GMAs. It will be extensively used throughout the article.
\begin{lem}
\label{gmalem}
Suppose $k$ is even, the $\omega_p^{1-k}$-component of the $p$-part of the class group of $\QQ(\zeta_p)$ is trivial and $\dim_{\FF_p}(H^1(G_{\QQ,p},\omega_p^{k-1})) =1$.
Let $R$ be a complete noetherian ring with residue field $\FF_p$ and $(t,d) : G_{\QQ,p\ell} \to R$ be a $p$-ordinary, $\ell$-unipotent pseudo-representation with determinant $\chi_p^{k-1}$ deforming $(\tr(\rhob_0),\det(\rhob_0))$. Then there exists a faithful GMA $A = \begin{pmatrix} R & B \\ C & R\end{pmatrix}$ over $R$ (in the sense of Definition~\ref{faithfuldef} above) and a representations $\rho : G_{\QQ,p\ell} \to A^{\times}$ such that
\begin{enumerate}
\item\label{bhaag1} $\tr(\rho) = t$, $\det(\rho) = d$ and $BC \subset m_R$,
\item\label{bhaag2} If $\rho(g) = \begin{pmatrix} a_g & b_g \\ c_g & d_g\end{pmatrix}$, then $a_g \equiv 1 \pmod{m_R}$ and $d_g \equiv \omega_p^{k-1}(g) \pmod{m_R}$ for all $g \in G_{\QQ,p\ell}$,
\item\label{bhaag3} $\rho(g_0) = \begin{pmatrix} a_0 & 0\\ 0 & d_0\end{pmatrix}$ and $R[\rho(G_{\QQ,p\ell})] =A$,
\item\label{bhaag4} $\rho|_{I_{\ell}}$ factors through the tame $\ZZ_p$-quotient of $I_{\ell}$, $\rho(i_{\ell}) = \begin{pmatrix} 1+x & b_{\ell} \\ c_{\ell} & 1-x \end{pmatrix}$ with $x \in m_R$ and $B=Rb_{\ell}$ i.e. $B$ is generated by $b_{\ell}$ as an $R$-module.
\item\label{bhaag5} There exists $c' \in R$ such that $C$ is generated by the set $\{c_{\ell},c'\}$ as an $R$-module.
\item\label{bhaag6} If $g \in I_p$, then $\rho(g)=\begin{pmatrix} 1 & 0\\c_g & \chi_{p}^{k-1}(g)\end{pmatrix}$.
\end{enumerate}
\end{lem}
\begin{proof}
The existence of $A = \begin{pmatrix} R & B \\ C & R\end{pmatrix}$ and $\rho : G_{\QQ,p\ell} \to A^{\times}$ satisfying parts \eqref{bhaag1}, \eqref{bhaag2} and \eqref{bhaag3} of the lemma follows directly from \cite[Proposition $2.4.2$]{Bel}. 
Moreover it also implies that $B$ and $C$ are finitely generated $R$-modules.
The description of $\rho(i_{\ell})$ follows from the assumption that $t(g)=2$ for all $g \in I_{\ell}$.
Since $(t|_{I_{\ell}},d|_{I_{\ell}})$ factors through the $\ZZ_p$-quotient of the tame inertia group at $\ell$ and $A$ is a faithful GMA, $\rho|_{I_{\ell}}$ also factors through this $\ZZ_p$-quotient of $I_{\ell}$.

If $B=0$ then faithfulness implies that $C=0$ and vice versa. All the parts of the lemma are clearly true in this case. So assume $B \neq 0$ and $C \neq 0$.

Let $h \in I_p$ and let $\rho(h) =\begin{pmatrix} a_h & b_h \\ c_h & d_h \end{pmatrix}$.
As $(t,d)$ is $p$-ordinary, we have for all $g \in G_{\QQ,p\ell}$ 
\begin{equation}\label{ordeq}
\tr(\rho(g)(\rho(g_0)-\chi_{p}^{k-1}(g_0))(\rho(h) -1))=0.
\end{equation}
As $R[\rho(G_{\QQ,p\ell})] =A$, we get that, for all $g' \in A$, $$\tr(g'(\rho(g_0)-\chi_{p}^{k-1}(g_0))(\rho(h) -1))=0.$$
For $c \in C$, let $g_c = \begin{pmatrix} 0 & 0\\ c &0\end{pmatrix} \in A$.
Then $$\tr(g_c(\rho(g_0)-\chi_{p}^{k-1}(g_0))(\rho(h) -1)) = (a_0 - \chi_{p}^{k-1}(g_0))b_hc.$$
As $a_0 \equiv 1 \pmod{m_R}$ and $\omega_p^{k-1}(g_0) \neq 1$, it follows that $a_0 - \chi_{p}^{k-1}(g_0) \in R^{\times}$ and hence, $b_hc=0$ for all $c \in C$.
As $A$ is faithful, we get that $b_h=0$ for all $h \in I_p$.

Taking $g'=\begin{pmatrix} 1 & 0\\ 0 & 0\end{pmatrix}$ in \eqref{ordeq}, we get $(a_0-\chi_{p}^{k-1}(g_0))(a_h-1)=0$ for all $h \in I_p$.
As $a_0 - \chi_{p}^{k-1}(g_0) \in R^{\times}$, we get that $a_h=1$ for all $h \in I_p$.
As $\det(\rho(h))=\chi_{p}^{k-1}(h)$, it follows that $d_h = \chi_{p}^{k-1}(h)$ for all $h \in I_p$.
This proves part~\eqref{bhaag6} of the lemma.

Let $B' := B/Rb_{\ell}$. Suppose $\phi : B'/m_RB' \to \FF_p$ is a map of $R$-modules. Then it induces a map $\phi^* : A \to M_2(\FF)$ of $R$-algebras which sends $\begin{pmatrix} a & b \\ c & d \end{pmatrix}$ to $\begin{pmatrix} a \pmod{m_R} & \phi(b) \\ 0 & d \pmod{m_R} \end{pmatrix}$. So the image of $\phi^*$ defines an element of $H^1(G_{\QQ,p\ell},\omega_p^{1-k})$.

Thus we get a map $f : \Hom(B'/m_RB',\FF_p) \to H^1(G_{\QQ,p\ell},\omega_p^{1-k})$ of $\FF_p$ vector spaces. It is easy to verify, using $R[\rho(G_{\QQ,p\ell})] =A$, that this map is injective (see the proofs of \cite[Theorem $1.5.5$]{BC} and \cite[Lemma $2.5$]{D2} for more details).
Note that if $x \in H^1(G_{\QQ,p\ell},\omega_p^{1-k})$ lies in the image of $f$, then part~\eqref{bhaag6} of the lemma implies that $x$ is unramified at $p$.
As $\rho|_{I_{\ell}}$ factors through the tame $\ZZ_p$-quotient of $I_{\ell}$, the definitions of $i_{\ell}$ and $b_{\ell}$ imply that $x$ is also unramified at $\ell$.
Since we are assuming that the $\omega_p^{1-k}$-component of the $p$-part of the class group of $\QQ(\zeta_p)$ is trivial, it follows that $x=0$. As $f$ is injective, we see that $B'=0$ which means $B=Rb_{\ell}$. This finishes the proof of part~\eqref{bhaag4} of the lemma.

Repeating the argument of the previous paragraph for $C'=C/Rc_{\ell}$, we get an injective map $f' : \Hom(C'/m_RC',\FF_p) \to H^1(G_{\QQ,p\ell},\omega_p^{k-1})$ of $\FF_p$ vector spaces.
Now $\rho|_{I_{\ell}}$ factors through the tame $\ZZ_p$-quotient of $I_{\ell}$.
Therefore, from the definitions of $i_{\ell}$ and $c_{\ell}$, it follows that if $x$ is in the image of $f'$, then $x$ is unramified at $\ell$. Hence, the image of $f'$ lies in $H^1(G_{\QQ,p}, \omega_p^{k-1}) \subset H^1(G_{\QQ,p\ell},\omega_p^{k-1})$. As we are assuming that $\dim(H^1(G_{\QQ,p}, \omega_p^{k-1}))=1$ and $f'$ is injective, we get that $C'$ is either $0$ or it is generated by one element as an $R$-module. This gives us part~\eqref{bhaag5} of the lemma.
\end{proof}

Reducible pseudo-representations deforming $(\tr(\rhob_0),\det(\rhob_0))$ will play an important role in the proofs of our main results.
Their importance is already highlighted in Wake's work (\cite{W}).
We will now prove a basic result about reducible pseudo-representations which is an analogue of \cite[Lemme $1$]{BC1}. Its proof is also similar to that of \cite[Lemme $1$]{BC1}. But we give it here for the benefit of reader.
We will use it extensively while working with reducible pseudo-representations.
\begin{lem}
\label{redlem}
Let $(t,d) :G_{\QQ,p\ell} \to R$ be a pseudo-representation deforming $(\tr(\rhob_0),\det(\rhob_0))$.
Suppose $A = \begin{pmatrix} R & B \\ C & R\end{pmatrix}$ is a (not necessarily faithful) GMA over $R$ and $\rho : G_{\QQ,p\ell} \to A^{\times}$ is a representation such that
\begin{enumerate}
\item $t=\tr(\rho)$ and $d=\det(\rho)$.
\item If $g \in G_{\QQ,p\ell}$ and $\rho(g) = \begin{pmatrix} a_g & b_g \\ c_g & d_g\end{pmatrix}$, then $a_g \equiv 1 \pmod{m_R}$ and $d_g \equiv \omega_p^{k-1}(g) \pmod{m_R}$.
\item $\rho(g_0) = \begin{pmatrix} a_0 & 0\\ 0 & d_0\end{pmatrix}$ and $R[\rho(G_{\QQ,p\ell})] =A$.
\end{enumerate}
Let $I$ be an ideal of $R$. Then $t \pmod{I}=\chi_1+\chi_2$, for some characters $\chi_1, \chi_2 : G_{\QQ,p\ell} \to (R/I)^{\times}$ deforming $1$ and $\omega_p^{k-1}$, if and only if $BC \subset I$.
Moreover, if this condition is satisfied, then $a_g \pmod{I}=\chi_1(g)$ and $d_g \pmod{I}=\chi_2(g)$ for all $g \in G_{\QQ,p\ell}$.
\end{lem}
\begin{proof}
It is easy to see, from the description of $\rho(g)$, that if $BC \subset I$, then $\tr(\rho) \pmod{I}=\chi_1+\chi_2$ for some characters $\chi_1$ and $\chi_2$ deforming $1$ and $\omega_p^{k-1}$. Indeed, we can take $\chi_1(g) = a_g \pmod{I}$ and $\chi_2(g)=d_g\pmod{I}$ for all $g \in G_{\QQ,p\ell}$.

Now suppose $\tr(\rho) \pmod{I}$ is a sum of two characters lifting $1$ and $\omega_p^{k-1}$.
If $r \in R$, then denote its image in $R/I$ by $\overline{r}$.
Suppose $g \in G_{\QQ,p\ell}$ and $\rho(g)=\begin{pmatrix} a_g & b_g \\ c_g & d_g\end{pmatrix}$. Then $t(gg_0) = a_0a_g+d_0d_g$.
By our assumption on $I$, we know that \begin{equation}\label{eq:2}\overline{a_0a_g+d_0d_g} = \chi_1(g_0g)+\chi_2(g_0g).\end{equation}
Now $d(g_0)=a_0d_0$ and hence, $a_0$ and $d_0$ are roots of the polynomial $f(X) = X^2-t(g_0)X+d(g_0) \in R[X]$.

Let $\bar f(X) \in R/I[X]$ be the reduction of $f$ modulo $I$.
So $\overline{a_0}$ and $\overline{d_0}$ are the roots of $\bar f(X)$. 
Now as $p >2$, $d(g)=\frac{t(g)^2-t(g^2)}{2}$ (see \cite[Section $1.4$]{BK}).
As $t \pmod{I}=\chi_1+\chi_2$, it follows that $\overline{t(g)}=\chi_1(g)+\chi_2(g)$ and $\overline{d(g)}=\chi_1(g)\chi_2(g)$.
So $\chi_1(g_0)$ and $\chi_2(g_0)$ are also roots of $\bar f(X)$.
Therefore, we get that $\overline{a_0}=\chi_1(g_0)$ and $\overline{d_0}=\chi_2(g_0)$ by matching their reductions modulo the maximal ideal of $R/I$.

Hence, by \eqref{eq:2}, we get $\overline{a_0a_g+d_0d_g} = \chi_1(g)\overline{a_0}+\chi_2(g)\overline{d_0}$.
On the other hand $\overline{a_g+d_g} = \chi_1(g) + \chi_2(g)$.
So, we get $\overline{a_g(d_0-a_0)}=\chi_1(g)\overline{(d_0-a_0)}$ and $\overline{d_g(a_0-d_0)}=\chi_2(g)\overline{(a_0-d_0)}$.
As $a_0-d_0 \in R^{\times}$, we get that $\overline{a_g}=\chi_1(g)$ and $\overline{d_g}=\chi_2(g)$. This proves the second part of the lemma.

Now if $g, g' \in G_{\QQ,p\ell}$, then $\rho(gg')=\rho(g)\rho(g')$. So we get $a_{gg'}=a_ga_{g'}+b_gc_{g'}$ and $d_{gg'}=d_gd_{g'}+c_gb_{g'}$.
From the previous paragraph, we know that $a_{gg'} \equiv a_ga_{g'} \pmod{I}$ and $d_{gg'} \equiv d_gd_{g'} \pmod{I}$.
Hence, for all $g, g' \in G_{\QQ,p\ell}$, we have $b_gc_{g'} \in I$ and $c_gb_{g'}\in I$. So we get $BC \subset I$ which proves the lemma.
\end{proof}

\subsection{First order deformations}
We will now focus on the $p$-ordinary, $\ell$-unipotent pseudo-representations $(t,d) : G_{\QQ,p\ell} \to \FF_p[\epsilon]/(\epsilon^2)$ with determinant $\omega_p^{k-1}$ deforming $(\tr(\rhob_0),\det(\rhob_0))$. 
Note that such pseudo-representations arise from the tangent space of $R^{\ps,\ord}_{\rhob_0,k}(\ell)/(p)$.
We start with determining the possible dimensions of the space of such deformations.
This will be useful in studying the structures of the Hecke algebras of interest and their Eisenstein ideals.

Let $(T,D) : G_{\QQ,p\ell} \to R^{\ps,\ord}_{\rhob_0,k}(\ell)$ be the universal $p$-ordinary, $\ell$-unipotent deformation of $(\tr(\rhob_0),\det(\rhob_0))$ with determinant $\chi_p^{k-1}$.
Recall, from \S\ref{notsec}, that we denote the tangent space of $R^{\ps,\ord}_{\rhob_0,k}(\ell)/(p)$ by $\tan(R^{\ps,\ord}_{\rhob_0,k}(\ell)/(p))$.
\begin{lem}
\label{tanlem}
$1 \leq \dim(\tan(R^{\ps,\ord}_{\rhob_0,k}(\ell)/(p))) \leq 2$.
\end{lem}
\begin{proof} 
Suppose  $\eta_1, \eta_2 : G_{\QQ,p\ell} \to (\FF_p[\epsilon]/(\epsilon^2))^{\times}$ are characters such that $\eta_1(I_p)=1$, $\eta_1(i_{\ell})=1+\epsilon$ and $\eta_2=\omega_p^{k-1}\eta_1^{-1}$.
By class field theory and definition of $i_{\ell}$, it follows that such characters exist and are unique.
Note that $\eta_1$ is a deformation of $1$, $\eta_2$ is a deformation of $\omega_p^{k-1}$ and $\eta_1\eta_2=\omega_p^{k-1}$.
It is easy to verify that the pseudo-representation $$(\eta_1+\eta_2, \eta_1\eta_2) : G_{\QQ,p\ell} \to \FF_p[\epsilon]/(\epsilon^2)$$ is a $p$-ordinary, $\ell$-unipotent deformation of $(\tr(\rhob_0),\det(\rhob_0))$ with determinant $\chi_p^{k-1}$.
Hence, we get that $1 \leq \dim(\tan(R^{\ps,\ord}_{\rhob_0,k}(\ell)/(p)))$.

On the other hand, suppose $(t,d) : G_{\QQ,p\ell} \to \FF_p[\epsilon]/(\epsilon^2)$ is a reducible, $p$-ordinary, $\ell$-unipotent deformation of $(\tr(\rhob_0),\det(\rhob_0))$ with determinant $\chi_p^{k-1}$. 
So there exist two characters $\chi_1, \chi_2 : G_{\QQ,p\ell} \to (\FF_p[\epsilon]/(\epsilon^2))^{\times}$ such that $\chi_1$ lifts $1$, $\chi_2$ lifts $\omega_p^{k-1}$ and $\chi_1\chi_2=\omega_p^{k-1}$.
From part~\eqref{bhaag6} of Lemma~\ref{gmalem}, we get that $\chi_1|_{I_p}=1$ and $\chi_2|_{I_p}=\omega_p^{k-1}$.

As $\chi_1$ is unramified at $p$, it follows, from class field theory and definition of $i_{\ell}$, that $\chi_1 = \eta_1^m$ and $\chi_2 =\omega_p^{k-1}\eta_1^{-m}$ for some integer $m \geq 0$.
So in the space of first order deformations of $(\tr(\rhob_0),\det(\rhob_0))$, the deformation $(t,d) = (\eta_1^m + \omega_p^{k-1}\eta_1^{-m}, \omega_p^{k-1})$ lies in the subspace generated by the deformation $(\eta_1+\eta_2,\eta_1\eta_2)$ found above.
So the space of reducible, $p$-ordinary, $\ell$-unipotent, first order deformations of $(\tr(\rhob_0),\det(\rhob_0))$ with determinant $\chi_p^{k-1}$ has dimension $1$. 

Let $\fm_0$ be the maximal ideal of $R^{\ps,\ord}_{\rhob_0,k}(\ell)$ and let $R:=R^{\ps,\ord}_{\rhob_0,k}(\ell)/(p,\fm_0^2)$.
So $\dim(\tan(R)) = \dim(\tan(R^{\ps,\ord}_{\rhob_0,k}(\ell)/(p)))$.
Denote the pseudo-representation $G_{\QQ,p\ell} \to R$ obtained by composing $(T,D)$ with the natural surjective map $R^{\ps,\ord}_{\rhob_0,k}(\ell) \to R$ by $(t,d)$.
Let $A=\begin{pmatrix} R & B\\ C & R\end{pmatrix}$ be the faithful GMA over $R$ and $\rho : G_{\QQ,p\ell} \to A^{\times}$ be the representation associated to $(t,d)$ by Lemma~\ref{gmalem}. If $BC=0$, then $(t,d)$ is reducible. Hence, we conclude, from the discussion above, that $\dim(\tan(R))=1$. Therefore, we get that $\dim(\tan(R^{\ps,\ord}_{\rhob_0,k}(\ell)/(p))) =1$.

Suppose $B \neq 0$ and $C \neq 0$. Now $\rho(i_{\ell}) = \begin{pmatrix}1+x & b_{\ell} \\ c_{\ell} & 1-x\end{pmatrix}$. As $d=\omega_p^{k-1}$, we see that $d(i_{\ell})=1$. This means $\det(\rho(i_{\ell}))=1$ which implies that $b_{\ell}c_{\ell}= -x^2$. As $x \in m_R$ and $m_R^2=0$, we get $b_{\ell}c_{\ell}=0$.
As $A$ is faithful, part~\eqref{bhaag4} of Lemma~\ref{gmalem} implies that $c_{\ell}=0$. 
Hence, by part~\eqref{bhaag5} of Lemma~\ref{gmalem}, it follows that $C$ is generated by $1$ element over $R$. 
Let $c$ be a generator of $C$ as $R$-module and let $x'=b_{\ell}c \in R$.
So $BC=(b_{\ell}c)=(x')$.

By Lemma~\ref{redlem}, we get that $(t \pmod{(x')}, d \pmod{(x')}) : G_{\QQ,p\ell} \to R/(x')$ is a reducible, $p$-ordinary, $\ell$-unipotent deformation of $(\tr(\rhob_0),\det(\rhob_0))$ with determinant $\chi_p^{k-1}$. 
So any first order deformation of $(\tr(\rhob_0),\det(\rhob_0))$ arising from $R/(x')$ is reducible.
Hence, from above, we conclude that $\dim(\tan(R/(x'))) \leq 1$. This implies that $\dim(\tan(R)) \leq 2$. Therefore, we get $\dim(\tan(R^{\ps,\ord}_{\rhob_0,k}(\ell)/(p))) \leq 2$ and the lemma follows.
\end{proof}

The next result will be used in determining the generators of the cotangent spaces of the Hecke algebras.
\begin{lem}
\label{tangenlem}
Let $R$ be a quotient of $R^{\ps,\ord}_{\rhob_0,k}(\ell)$ and $(t,d) : G_{\QQ,p\ell} \to R$ be the pseudo-representation obtained by composing $(T,D)$ with the quotient map $R^{\ps,\ord}_{\rhob_0,k}(\ell) \to R$. 
Let $A$ be the faithful GMA over $R$ and $\rho : G_{\QQ,p\ell} \to A^{\times}$ be the representation attached to $(t,d)$ by Lemma~\ref{gmalem}.
Let $\rho(i_{\ell}) = \begin{pmatrix} 1 + x & b_{\ell} \\ c_{\ell} & 1-x \end{pmatrix}$.
If $\dim(\tan(R/(p)))=1$ and the deformation $(t',d') : G_{\QQ,p\ell} \to R/(p,m_R^2)$ of $(\tr(\rhob_0),\det(\rhob_0))$ obtained by composing $(t,d)$ with the quotient map $R \to R/(p,m_R^2)$ is reducible, then $m_R$ is generated by $p$ and $x$.
\end{lem}
\begin{proof}
As $\dim(\tan(R/(p)))=1$, we see that $R/(p,m_R^2) \simeq \FF_p[\epsilon]/(\epsilon^2)$. So $(t',d')$ gives us a non-trivial $\FF_p[\epsilon]/(\epsilon^2)$-valued pseudo-representation. We will now use this identification.
Since $(t',d')$ is reducible, $t' = \chi_1 + \chi_2$, where $\chi_1, \chi_2 : G_{\QQ,p\ell} \to (\FF_p[\epsilon]/(\epsilon^2))^{\times}$ are characters deforming $1$ and $\omega_p^{k-1}$, respectively.
From the proof of Lemma~\ref{tanlem}, we know that $\chi_1|_{I_p} = 1$, $\chi_2|_{I_p}=\omega_p^{k-1}$, $\chi_1(i_{\ell}) = 1 + a\epsilon$ and $\chi_2(i_{\ell}) = 1 - a\epsilon$ for some non-zero $a \in \FF_p$.
From Lemma~\ref{redlem}, we get that the image of $1+x$ in $R/(p,m_R^2) \simeq \FF_p[\epsilon]/(\epsilon^2)$ is $1+a\epsilon$. Therefore the image of $x$ in $R/(p,m_R^2) \simeq \FF_p[\epsilon]/(\epsilon^2)$  generates the ideal $(\epsilon)$.
Hence, we conclude that $m_R$ is generated by $(p,x)$.
\end{proof}

Recall that we have chosen a non-zero generator $b_0$ of $H^1(G_{\QQ,p},\omega_p^{k-1})$ and have denoted the cup product of $c_0$ and $b_0$ by $c_0 \cup b_0$. So $c_0 \cup b_0 \in H^2(G_{\QQ,p\ell},1)$.
We now give a necessary condition, in terms of this cup product, for the existence of a first order deformation of $(\tr(\rhob_0),\det(\rhob_0))$ which is $p$-ordinary, $\ell$-unipotent with determinant $\omega_p^{k-1}$ and is not reducible.
This lemma is the first step towards establishes the link between the principality of the Eisenstein ideal and the non-vanishing of the cup product $c_0 \cup b_0$.
The proof uses techniques similar to the ones used in \cite{B1}.

\begin{lem}
\label{cuplem}
If $\dim(\tan(R^{\ps,\ord}_{\rhob_0,k}(\ell)/(p))) = 2$, then $c_0 \cup b_0 =0$.
\end{lem}
\begin{proof}
If $\dim(\tan(R^{\ps,\ord}_{\rhob_0,k}(\ell)/(p))) = 2$, then the proof of Lemma~\ref{tanlem} implies that there exists a pseudo-representation $(t,d) : G_{\QQ,p\ell} \to \FF_p[\epsilon]/(\epsilon^2)$ such that $(t,d)$ is a $p$-ordinary, $\ell$-unipotent deformation of $(\tr(\rhob_0),\det(\rhob_0))$ with determinant $\omega_p^{k-1}$ and $(t,d)$ is \emph{not} reducible. Let $(t_0,d_0)$ be such a deformation.

Let $A= \begin{pmatrix} \FF_p[\epsilon]/(\epsilon^2) & B \\ C & \FF_p[\epsilon]/(\epsilon^2) \end{pmatrix}$ be the faithful GMA over $\FF_p[\epsilon]/(\epsilon^2)$ and $\rho : G_{\QQ,p\ell} \to A^{\times}$ be the representation attached to $(t_0,d_0)$ by Lemma~\ref{gmalem}. So both $B \neq 0$ and $C \neq 0$.
Now we know, from part~\eqref{bhaag4} of Lemma~\ref{gmalem}, that $\rho(i_{\ell}) = \begin{pmatrix} 1+x & b_{\ell} \\ c_{\ell} & 1-x\end{pmatrix}$ with $x \in (\epsilon)$.
As $d_0(i_{\ell})=1$, we get that $b_{\ell}c_{\ell}=-x^2=0$.
We also know, from part~\eqref{bhaag4} of Lemma~\ref{gmalem}, that $B = \FF_p[\epsilon]/(\epsilon^2).b_{\ell}$. Since $B \neq 0$, we get that $b_{\ell} \neq 0$ and hence, $c_{\ell}=0$.
So we conclude, from part~\eqref{bhaag5} of Lemma~\ref{gmalem}, that $C$ is also generated by one element over $\FF_p[\epsilon]/(\epsilon^2)$. 

Let $c$ be a generator of $C$.
As $BC \subset (\epsilon)$, we get $\epsilon BC=0$. As $A$ is faithful, we get that $\epsilon B =0$ and $\epsilon C=0$.
Hence, both $B$ and $C$ are isomorphic to $\FF_p$ as $\FF_p[\epsilon]/(\epsilon^2)$-modules.
The choice of generators $b_{\ell}$ for $B$ and $c$ for $C$ identifies both $B$ and $C$ with $\FF_p$.
In particular, if $g \in G_{\QQ,p\ell}$, then $\rho(g) = \begin{pmatrix} 1+a_g\epsilon & b_gb_{\ell} \\ c_gc &\omega_p^{k-1}(g) + d_g\epsilon\end{pmatrix}$ with $a_g, b_g, c_g, d_g \in \FF_p$.

Since $\rho$ is a representation and $\epsilon b_{\ell} = \epsilon c =0$, we get that the map $G_{\QQ,p\ell} \to \FF_p$ sending $g$ to $c_g$ defines an element of $H^1(G_{\QQ,p\ell},\omega_p^{k-1})$ and the map $G_{\QQ,p\ell} \to \FF_p$ sending $g$ to $\omega_p^{1-k}(g)b_g$ defines an element of $H^1(G_{\QQ,p\ell},\omega_p^{1-k})$. As $\FF_p[\epsilon]/(\epsilon^2)[\rho(G_{\QQ,p\ell})] = A$, it is easy to verify (using proof of Lemma ~\ref{gmalem}) that both these elements are non-zero.

Note that part~\eqref{bhaag6} of Lemma~\ref{gmalem} implies that $b_h=0$ for all $h \in I_p$. Hence, the cohomology class defined by $b_g$ lies in $\ker(H^1(G_{\QQ,p\ell},\omega_p^{1-k}) \to H^1(I_p,\omega_p^{1-k}))$. 
Since $c_{\ell}=0$ and $\rho|_{I_{\ell}}$ factors through the tame $\ZZ_p$-quotient of $I_{\ell}$ (part~\eqref{bhaag4} of Lemma~\ref{gmalem}), the definition of $i_{\ell}$ implies that the cohomology class defined by $c_g$ lies in $H^1(G_{\QQ,p},\omega_p^{k-1})$.

Since both these spaces are one dimensional and generated by $c_0$ and $b_0$, respectively, it follows that there exist non-zero elements $\alpha, \beta \in \FF_p$ such that $c_g = \alpha b_0(g)$ and $b_g = \beta\omega_p^{k-1}(g)c_0(g)$ for all $g \in G_{\QQ,p\ell}$. Let $\gamma \in \FF_p$ such that $b_{\ell}c=\gamma\epsilon$. So $\gamma \neq 0$.
Now we have $$1+a_{gh}\epsilon = (1+a_g\epsilon)(1+a_h\epsilon) + b_gc_h\gamma\epsilon = (1+a_g\epsilon)(1+a_h\epsilon) +\alpha\beta\gamma\omega_p^{k-1}(g)c_0(g)b_0(h)\epsilon.$$
 So we have $a_{gh} - a_g - a_h = \alpha\beta\gamma\omega_p^{k-1}(g)c_0(g)b_0(h)$ for all $g, h \in G_{\QQ,p\ell}$.
As $\alpha\beta\gamma \neq 0$, it follows, from the definition of the cup product, that $c_0 \cup b_0 = 0$ (see \cite[Section $2.1$]{SS}).
\end{proof}

We will now give a criteria for the vanishing of the cup product $c_0 \cup b_0$.

\begin{lem}
\label{vanishlem}
Suppose Vandiver's conjecture holds for $p$.
Then $c_0 \cup b_0 = 0$ if and only if $\prod_{j=1}^{p-1}(1-\zeta_p^j)^{j^{k-2}}  \in (\ZZ/\ell\ZZ)^{\times}$ is a $p$-th power in $(\ZZ/\ell\ZZ)^{\times}$, where $\zeta_p \in (\ZZ/\ell\ZZ)^{\times}$ is a primitive $p$-th root of unity.
\end{lem}
\begin{proof}
By \cite[Proposition $2.4.1$]{SS}, we know that $c_0 \cup b_0 = 0$ if and only if $c_0|_{G_{\QQ_{\ell}}} \cup b_0|_{G_{\QQ_{\ell}}} = 0$.
Note that $c_0|_{G_{\QQ_{\ell}}}, b_0|_{G_{\QQ_{\ell}}} \in H^1(G_{\QQ_{\ell}},1)$. Since $c_0|_{G_{\QQ_{\ell}}}$ is ramified at $\ell$, we see that $c_0|_{G_{\QQ_{\ell}}} \neq 0$. On the other hand, $b_0|_{G_{\QQ_{\ell}}}$ is unramified at $\ell$.
So $c_0|_{G_{\QQ_{\ell}}} \cup b_0|_{G_{\QQ_{\ell}}} = 0$ if and only if $b_0|_{G_{\QQ_{\ell}}} = 0$.

Now let $\rhob_{b_0} : G_{\QQ,p}\to \GL_2(\FF_p)$ be the representation given by $\begin{pmatrix} \omega_p^{k-1} & * \\ 0 & 1\end{pmatrix}$, where $*$ corresponds to $b_0$ and let $K'$ be the extension of $\QQ$ fixed by $\ker(\rhob_{b_0})$.
So $b_0|_{G_{\QQ_{\ell}}} = 0$ if and only if $\ell$ splits completely in $K'$.
As $p \mid \ell-1$, $\ell$ splits completely in $K'$ if and only if $\ell$ splits completely in $K'(\zeta_p)$.

Let $\mathcal{U} = \ZZ[\zeta_p,p^{-1}]^{\times}$. So $\dfrac{\mathcal{U}}{\mathcal{U}^p}$ is a $\text{Gal}(\QQ(\zeta_p)/\QQ)$-module.  
For any integer $i$, let $\dfrac{\mathcal{U}}{\mathcal{U}^p}[\omega_p^i]$ be the $\omega_p^i$-component of $\dfrac{\mathcal{U}}{\mathcal{U}^p}$.
Using the inflation-restriction exact sequence and Kummer theory, we get that $\dfrac{\mathcal{U}}{\mathcal{U}^p}[\omega_p^{2-k}]$ is isomorphic to a subgroup of $H^1(G_{\QQ,p},\omega_p^{k-1})$.
Since Vandiver's conjecture holds for $p$, we know, from Remark~\ref{rem2}, that $\dim(H^1(G_{\QQ,p},\omega_p^{k-1}))=1$ (which is consistent with the hypothesis~\eqref{teen} of Setup~\ref{basic}).
Observe that $\dfrac{\mathcal{U}}{\mathcal{U}^p}[\omega_p^{2-k}]$ is an $\FF_p$-vector space of dimension $1$ and hence, $H^1(G_{\QQ,p},\omega_p^{k-1}) \simeq \dfrac{\mathcal{U}}{\mathcal{U}^p}[\omega_p^{2-k}]$.
We refer the reader to the discussion in \cite[Section $5.3$]{SS} appearing just before \cite[Theorem $5.3.2$]{SS} for more details.

Hence, if $\Xi \in \mathcal{U}$ is an element such that $\Xi \not\in \mathcal{U}^p$ and $g(\Xi) \equiv \Xi^{\omega_p^{2-k}(g)} \pmod{\mathcal{U}^p}$ for all $g \in \text{Gal}(\QQ(\zeta_p)/\QQ)$, then $K'(\zeta_p)$ is obtained by attaching a $p$-th root of $\Xi$ to $\QQ(\zeta_p)$.

Now let $\Xi= \prod_{j=1}^{p-1}(1-\zeta_p^j)^{j^{k-2}}$.
Observe that $g(\Xi) \equiv \Xi^{\omega_p^{2-k}(g)} \pmod{\mathcal{U}^p}$ for all $g \in \text{Gal}(\QQ(\zeta_p)/\QQ)$.
Recall that, by our assumption, Vandiver's conjecture holds for $p$.
Therefore, by combining \cite[Lemma 2.7]{L}, \cite[Lemma 8.1]{Wash} and \cite[Theorem 8.2]{Wash}, we get that the set $\{1-\zeta_p^a \mid a \in \ZZ, 0 < a < p/2\}$ is a $\ZZ$-basis of the free part of $\mathcal{U}$.
As $k$ is even, this implies that $\Xi \in \mathcal{U}\setminus \mathcal{U}^p$.
Hence, we conclude that the extension $K'(\zeta_p)$ is obtained by attaching a $p$-th root of $\prod_{j=1}^{p-1}(1-\zeta_p^j)^{j^{k-2}}$ to $\QQ(\zeta_p)$ (see also \cite[Remark $3.2.1$]{W}).

So $\ell$ splits completely in $K'(\zeta_p)$ if and only if $\prod_{j=1}^{p-1}(1-\zeta_p^j)^{j^{k-2}}$ is a $p$-th power in $\QQ_{\ell}$.
By Hensel's lemma, $\prod_{j=1}^{p-1}(1-\zeta_p^j)^{j^{k-2}}$ is a $p$-th power in $\QQ_{\ell}$ if and only if $\prod_{j=1}^{p-1}(1-\zeta_p^j)^{j^{k-2}}  \in (\ZZ/\ell\ZZ)^{\times}$ is a $p$-th power in $(\ZZ/\ell\ZZ)^{\times}$ which proves the lemma.
\end{proof}

\subsection{Pseudo-representations and representations}
\label{repsubsec}
Let $R^{\defo,\ord}_{\rhob_{c_0},k}(\ell)$ be the deformation ring introduced in Definition~\ref{ordinarydef} and $\rho^{\univ, \ell} : G_{\QQ,p\ell} \to \GL_2(R^{\defo,\ord}_{\rhob_{c_0},k}(\ell))$ be the corresponding deformation.
Note that $(\tr(\rho^{\univ,\ell}),\det(\rho^{\univ,\ell})) : G_{\QQ,p\ell} \to R^{\defo,\ord}_{\rhob_{c_0},k}(\ell)$ is a deformation of $(\tr(\rhob_0),\det(\rhob_0))$.
Note that $\det(\rho^{\univ,\ell}) = \chi_{p}^{k-1}$.

As $\rho^{\univ,\ell}$ is $p$-ordinary, it follows that, under a suitable basis, $\rho^{\univ,\ell}(g) = \begin{pmatrix} \chi_{p}(g)^{k-1} & b_g \\ 0 & 1 \end{pmatrix}$ for all $g \in I_p$.
So $(\rho^{\univ,\ell}(g) - \chi_{p}(g)^{k-1})(\rho^{\univ,\ell}(h)-1) = \begin{pmatrix} 0 & 0 \\ 0 & 0 \end{pmatrix}$ for all $g,h \in I_p$.
Hence, $(\tr(\rho^{\univ,\ell}),\det(\rho^{\univ,\ell}))$ is $p$-ordinary. Moreover, it is $\ell$-unipotent by definition.

Therefore, $(\tr(\rho^{\univ,\ell}),\det(\rho^{\univ,\ell})) $ induces a map $\phi : R^{\ps,\ord}_{\rhob_0,k}(\ell) \to R^{\defo,\ord}_{\rhob_{c_0},k}(\ell)$.

\begin{lem}
\label{surjlem}
The map $\phi : R^{\ps,\ord}_{\rhob_0,k}(\ell) \to R^{\defo,\ord}_{\rhob_{c_0},k}(\ell)$ induced by $(\tr(\rho^{\univ,\ell}),\det(\rho^{\univ,\ell})) $ is surjective.
\end{lem}
\begin{proof}
To prove the lemma, it suffices to prove that if $\rho : G_{\QQ,p\ell} \to \GL_2(\FF_p[\epsilon]/(\epsilon^2))$ is an ordinary deformation of $\rhob_{c_0}$ such that $\det(\rho)=\omega_p^{k-1}$ and $\tr(\rho) = 1 + \omega_p^{k-1}$, then $\rho \simeq \rhob_{c_0}$.
Denote $\FF_p[\epsilon]/(\epsilon^2)$ by $R$ and let $\rho$ be such a representation.
By \cite[Lemma $3.1$]{D3}, after replacing $\rho$ by a suitable element in its equivalence class if necessary, we can assume that $\rho(g_0)=\begin{pmatrix}a_0 & 0\\ 0 & d_0\end{pmatrix}$. 
Since $a_0 \neq d_0$, \cite[Lemma $2.4.5$]{Bel} implies that there exist ideals $B, C \subset R$ such that $R[\rho(G_{\QQ,p\ell})] = \begin{pmatrix}R & B\\ C & R\end{pmatrix}$. 

Note that $R[\rho(G_{\QQ,p\ell})]$ is a GMA over $R$ with multiplication of B and C given by multiplication in $R$.
As $\rho \pmod{(\epsilon)} = \rhob_{c_0}$, it follows that $B=R$.
Now $(\tr(\rho),\det(\rho))$ is reducible. Hence, applying Lemma~\ref{redlem} to the GMA $R[\rho(G_{\QQ,p\ell})]$, we get that $BC=0$. As $B=R$, we have $C=0$.
Moreover, Lemma~\ref{redlem} also implies that $\rho(g) = \begin{pmatrix}1 & b_g\\ 0 & \omega_p^{k-1}(g)\end{pmatrix}$ for all $g \in G_{\QQ,p\ell}$. 

So if $f_1,f_2 : G_{\QQ,p\ell} \to \FF_p$ are functions such that $b_g = f_1(g) + \epsilon(f_2(g))$, then $\omega_p^{1-k}f_1, \omega_p^{1-k}f_2 \in H^1(G_{\QQ,p\ell},\omega_p^{1-k})$.
As $\rho$ is $p$-ordinary, it follows that $\rho|_{I_p} \simeq 1 \oplus \omega_p^{k-1}$.
Hence, by changing the basis if necessary, we can assume that $b_h=0$ for all $h \in I_p$.

Thus, we see that $\omega_p^{1-k}f_1, \omega_p^{1-k}f_2 \in \ker(H^1(G_{\QQ,p\ell},\omega_p^{1-k}) \to H^1(I_p,\omega_p^{1-k}))$.
From Lemma~\ref{cohomlem}, we know that $\ker(H^1(G_{\QQ,p\ell},\omega_p^{1-k}) \to H^1(I_p,\omega_p^{1-k}))$ is generated by $c_0$.
Hence, there exist $\alpha, \beta \in \FF_p$ such that $\omega_p^{1-k}f_1 = \alpha c_0$ and $\omega_p^{1-k}f_2 = \beta c_0$.
Note that $\alpha = 1$ as $\rho \pmod{(\epsilon)} = \rhob_{c_0}$.
So conjugating $\rho$ by $\begin{pmatrix} (1+ \epsilon\beta)^{-1} & 0 \\ 0 & 1\end{pmatrix}$ gives us $\rhob_{c_0}$ which proves the lemma.
\end{proof}

\begin{defi}
\begin{enumerate}
\item Let $I_0$ be the ideal of $R^{\defo,\ord}_{\rhob_{c_0},k}(\ell)$ generated by the set $\{(\tr(\rho^{\univ,\ell}(\frob_q)) - (1+q^{k-1}) \mid q \text{ is a prime } \neq p,\ell \}$. We call it the Eisenstein ideal of $R^{\defo,\ord}_{\rhob_{c_0},k}(\ell)$.
\item Let $J_0$ be the ideal of $R^{\ps,\ord}_{\rhob_0,k}(\ell)$ generated by the set $\{T(\frob_q) - (1+q^{k-1}) \mid q \text{ is a prime } \neq p,\ell \}$. We call it the Eisenstein ideal of $R^{\ps,\ord}_{\rhob_0,k}(\ell)$.
\end{enumerate}
\end{defi}
We will now give an upper bound on the index of $I_0$ in $R^{\defo,\ord}_{\rhob_{c_0},k}(\ell)$ in terms of $\nu$ and $v_p(k)$ introduced in \S\ref{notsec}.

\begin{lem}
\label{eisenlem}
The ideal $I_0$ has finite index in $R^{\defo,\ord}_{\rhob_{c_0},k}(\ell)$ and $R^{\defo,\ord}_{\rhob_{c_0},k}(\ell)/I_0$ is cyclic of order at most $ p^{\nu+v_p(k)}$.
\end{lem}
\begin{proof}
Note that the ideal generated by $\phi(J_0)$ in $R^{\defo,\ord}_{\rhob_{c_0},k}(\ell)$ is $I_0$ where $\phi : R^{\ps,\ord}_{\rhob_0,k}(\ell) \to R^{\defo,\ord}_{\rhob_{c_0},k}(\ell)$ is the surjective map obtained in Lemma~\ref{surjlem}.
By definition of $J_0$, it follows that $R^{\ps,\ord}_{\rhob_0,k}(\ell) /J_0 \simeq \ZZ_p$. Indeed, it is the kernel of the map $R^{\ps,\ord}_{\rhob_0,k}(\ell) \to \ZZ_p$ induced by the pseudo-representation $(1+\chi_p^{k-1},\chi_p^{k-1})$.
So $R^{\defo,\ord}_{\rhob_{c_0},k}(\ell)/I_0$ is a quotient of $\ZZ_p$ and hence, it is cyclic.

Let $R=R^{\defo,\ord}_{\rhob_{c_0},k}(\ell)/I_0$ and $\rho : G_{\QQ,p\ell} \to \GL_2(R)$ be the representation obtained by composing $\rho^{\univ,\ell}$ with the natural surjective map $R^{\defo,\ord}_{\rhob_{c_0},k}(\ell) \to R^{\defo,\ord}_{\rhob_{c_0},k}(\ell)/I_0$.
Note that $\tr(\rho) = 1 + \chi_{p}^{k-1}$ and $\det(\rho)=\chi_{p}^{k-1}$.
Using the arguments of Lemma~\ref{surjlem}, we conclude that (after replacing $\rho$ with a representation in its equivalence class if necessary) $\rho = \begin{pmatrix} 1 & * \\ 0 & \chi_{p}^{k-1}\end{pmatrix}$, where $*$ is non-zero and is unramified at $p$.

As $\rho$ lifts $\rhob_{c_0}$, it follows, from the definition of $i_{\ell}$, that $\rho(i_{\ell}) = \begin{pmatrix} 1 & b_{\ell} \\ 0 & 1\end{pmatrix}$ and $b_{\ell} \in (R^{\defo,\ord}_{\rhob_{c_0},k}(\ell)/I_0)^{\times}$.
Let $g_{\ell} \in G_{\QQ_{\ell}}$ be a lift of $\frob_{\ell}$.
As $\rho(g_{\ell}i_{\ell}g_{\ell}^{-1})=\rho(i_{\ell})^{\ell}$, we get that $(\ell^{1-k}-\ell)b_{\ell}=0$.
Since $b_{\ell}$ is a unit, we get that $\ell^k-1=0$.
Hence it follows that $R/I_0$ is finite (as it is a quotient of $\ZZ_p$) and $|R/I_0| \leq |\ZZ_p/(\ell^k-1)\ZZ_p|$.
Now the highest power of $p$ dividing $\ell^k-1$ is $\nu+v_p(k)$ and the lemma follows from this.
\end{proof}

 We will now give a necessary and sufficient condition for the existence of a reducible, $p$-ordinary first order deformation of $\rhob_{c_0}$ with determinant $\omega_p^{k-1}$.
Before proceeding further, let us develop some notation.

Let $\text{Ad}(\rhob_{c_0})$ be the adjoint representation of $\rhob_{c_0}$. So it is the space of $2 \times 2$ matrices over $\FF_p$ on which $g \in G_{\QQ,p\ell}$ acts by conjugation by $\rhob_{c_0}(g)$. 
Let $\text{Ad}^0(\rhob_{c_0})$ be the subspace of trace $0$ matrices of $\text{Ad}(\rhob_{c_0})$ and $V$ be the subspace of $\text{Ad}^0(\rhob_{c_0})$ given by upper triangular matrices.

It is easy to verify that $V$ is a $G_{\QQ,p\ell}$-subrepresentation of $\text{Ad}^0(\rhob_{c_0})$ and it is isomorphic to $\rhob'_{c_0} := \rhob_{c_0} \otimes \omega_p^{1-k}$.
Note that $\text{Ad}^0(\rhob_{c_0})/\rhob'_{c_0} = \omega_p^{k-1}$. So the natural map $H^1(G_{\QQ,p\ell}, \rhob'_{c_0}) \to H^1(G_{\QQ,p\ell}, \text{Ad}^0(\rhob_{c_0}))$ is injective.

By class field theory, we know that $\dim(\ker(H^1(G_{\QQ,p\ell},1) \to H^1(I_p,1))) = 1$. Recall that we have chosen a generator $a_0$ of $\ker(H^1(G_{\QQ,p\ell},1) \to H^1(I_p,1))$ and have denoted by $c_0 \cup a_0$ the cup product of $c_0$ and $a_0$. So $c_0 \cup a_0 \in H^2(G_{\QQ,p\ell}, \omega_p^{1-k})$.

Let $R$ be an object of $\mathcal{C}$.
We say that a deformation $\rho : G_{\QQ,p\ell} \to \GL_2(R)$ of $\rhob_{c_0}$ is reducible if there exist characters $\chi_1, \chi_2 : G_{\QQ,p\ell} \to R^{\times}$ deforming $1$ and $\omega_p^{k-1}$, respectively such that $\tr(\rho) = \chi_1 + \chi_2$.

We will now prove one of the key lemmas which link first order reducible deformations of $(\tr(\rhob_0),\det(\rhob_0))$ with the vanishing of the cup product $c_0 \cup a_0$.

\begin{lem}
\label{inflem}
There exists a $p$-ordinary deformation $\rho : G_{\QQ,{p\ell}} \to \GL_2(\FF_p[\epsilon]/(\epsilon^2))$ of $\rhob_{c_0}$ with determinant $\omega_p^{k-1}$ which is reducible and not isomorphic to $\rhob_{c_0}$ if and only if $c_0 \cup a_0 =0$.
\end{lem}
\begin{proof}
Let $\rho : G_{\QQ,{p\ell}} \to \GL_2(\FF_p[\epsilon]/(\epsilon^2))$ be a deformation of $\rhob_{c_0}$ with determinant $\omega_p^{k-1}$.
Let $R=\FF_p[\epsilon]/(\epsilon^2)$.
From \cite[Lemma $3.1$]{D3}, we can assume, by changing the basis if necessary, that $\rho(g_0)=\begin{pmatrix} a_0 &0\\ 0 & d_0\end{pmatrix}$ with $a_0 \pmod{(\epsilon)} =1 $ and $d_0 \pmod{(\epsilon)} = \omega_p^{k-1}(g_0)$.

Now \cite[Lemma $2.4.5$]{Bel} implies that there exists an ideal $ C \subset R$ such that $R[\rho(G_{\QQ,p\ell})] = \begin{pmatrix}R & R\\ C & R\end{pmatrix}$. 
By Lemma~\ref{redlem}, it follows that $\rho$ is reducible if and only if $C=0$ and in this case, $\rho \simeq \begin{pmatrix} \chi_1 & *\\ 0 & \chi_2\end{pmatrix}$, where $\chi_1$ and $\chi_2$ are characters of $G_{\QQ,p\ell}$ deforming $1$ and $\chi_p^{k-1}$, respectively.

Note that $\rho$ corresponds to an element of $x \in H^1(G_{\QQ,p\ell},\text{Ad}^0(\rhob_{c_0}))$. So from the previous paragraph, it follows that $\rho$ is reducible if and only if $x \in  H^1(G_{\QQ,p\ell},\rhob'_{c_0}) \subset  H^1(G_{\QQ,p\ell},\text{Ad}^0(\rhob_{c_0}))$.
Since $\chi_1$ is a lift of $1$ and $\chi_2$ is a lift of $\omega_p^{k-1}$, it follows that if $\rho$ is reducible, then $\rho$ is $p$-ordinary if and only if $\rho|_{I_p} \simeq 1 \oplus \omega_p^{k-1}$.

So $\rho$ is reducible and $p$-ordinary if and only if $x \in \ker(H^1(G_{\QQ,p\ell},\rhob'_{c_0}) \to H^1(I_p,\rhob'_{c_0}))$.
Thus there exists a non-trivial, reducible, $p$-ordinary deformation $\rho : G_{\QQ,p\ell} \to \GL_2(\FF_p[\epsilon]/(\epsilon^2))$ of $\rhob_{c_0}$ with determinant $\omega_p^{k-1}$ if and only if $\ker(H^1(G_{\QQ,p\ell},\rhob'_{c_0}) \to H^1(I_p,\rhob'_{c_0})) \neq 0$.

Note that an element $x$ of $H^1(G_{\QQ,p\ell},\rhob'_{c_0})$ gives a representation $\rho'_x : G_{\QQ,p\ell} \to \GL_3(\FF_p)$ such that 
$$\rho'_x(g) = \begin{pmatrix} \omega_p^{1-k}(g) & c_0(g) & F(g) \\ 0 & 1 & b(g) \\ 0 & 0 & 1\end{pmatrix} \text{ for all } g \in G_{\QQ,p\ell}.$$
Note that $b \in H^1(G_{\QQ,p\ell},1)$.
So $x \in \ker(H^1(G_{\QQ,p\ell},\rhob'_{c_0}) \to H^1(I_p,\rhob'_{c_0}))$ if and only if $b(I_p)=0$ and $F(I_p)=0$ in the corresponding $\rho'_x$ (after changing the basis if necessary).
Hence, $\ker(H^1(G_{\QQ,p\ell},\rhob'_{c_0}) \to H^1(I_p,\rhob'_{c_0})) \neq 0$ if and only if there exists a representation $\rho' : G_{\QQ,p\ell} \to \GL_3(\FF_p)$ such that $$\rho'(g) = \begin{pmatrix} \omega_p^{1-k}(g) & c_0(g) & F(g) \\ 0 & 1 & a_0(g) \\ 0 & 0 & 1\end{pmatrix} \text{ for all } g \in G_{\QQ,p\ell} \text{ and } F(I_p)=0.$$

Now if such a $\rho'$ exists, then it is easy to verify that the coboundary of $-F : G_{\QQ,p\ell} \to \FF_p$ is $c_0 \cup a_0$ and hence, $c_0 \cup a_0 =0$.

On the other hand, suppose $c_0 \cup a_0=0$ and let $F : G_{\QQ,p\ell} \to \FF_p$ be the map such that coboundary of $-F$ is $c_0 \cup a_0$.
Since $c_0$ is unramified at $p$ and $\omega_p^{1-k}|_{G_{\QQ_p}} \neq 1$, it follows that $c_0 |_{G_{\QQ_p}} =0$. Hence, $F|_{G_{\QQ_p}} \in H^1(G_{\QQ_p},\omega_p^{1-k})$.

Since we are assuming that the $\omega_p^{1-k}$-component of the $p$-part of the class group of $\QQ(\zeta_p)$ is trivial, we know that $\ker(H^1(G_{\QQ,p},\omega_p^{1-k}) \to H^1(G_{\QQ_p},\omega_p^{1-k})) =0$. 
From the proof of Lemma~\ref{cohomlem}, we know that $\dim(H^1(G_{\QQ,p},\omega_p^{1-k})) = \dim(H^1(G_{\QQ_p},\omega_p^{1-k})) =1$. 
So the restriction map $H^1(G_{\QQ,p} , \omega_p^{1-k}) \to H^1(G_{\QQ_p},\omega_p^{1-k})$ is an isomorphism.
Hence, we can change $F$ by a suitable element of $H^1(G_{\QQ,p} , \omega_p^{1-k})$ to assume that $F(I_p)=0$.

This means that there exists a representation $\rho' : G_{\QQ,p\ell} \to \GL_3(\FF_p)$ such that
$\rho'(g) = \begin{pmatrix} \omega_p^{1-k}(g) & c_0(g) & F(g) \\ 0 & 1 & a_0(g) \\ 0 & 0 & 1\end{pmatrix}$ for all $g \in G_{\QQ,p\ell}$ and $F(I_p)=0$. This completes the proof of the lemma.
\end{proof}

We say that a pseudo-representation $(t,d) : G_{\QQ,p\ell} \to R$ arises from a representation if there exists a representation $\rho : G_{\QQ,p\ell} \to \GL_2(R)$ such that $\tr(\rho) =t$ and $\det(\rho) = d$.

\begin{lem}
\label{infinlem}
Suppose $(t,d) : G_{\QQ,p\ell} \to \FF_p[\epsilon]/(\epsilon^2)$ is a non-trivial, reducible, $p$-ordinary, $\ell$-unipotent pseudo-representation with determinant $\omega_p^{k-1}$ deforming $(\tr(\rhob_0),\det(\rhob_0))$. Then $(t,d)$ arises from a $p$-ordinary representation deforming $\rhob_{c_0}$ if and only if $c_0 \cup a_0=0$.
\end{lem}
\begin{proof}
Suppose $\rho : G_{\QQ,p\ell} \to \GL_2(\FF_p[\epsilon]/(\epsilon^2))$ is a $p$-ordinary deformation of $\rhob_{c_0}$ such that $t=\tr(\rho)$ and $d=\det(\rho)$.
As $(t,d)$ is non-trivial, reducible and $d=\omega_p^{k-1}$, Lemma~\ref{inflem} implies that $c_0 \cup a_0 =0$.

Now suppose $c_0 \cup a_0 =0$. Then Lemma~\ref{inflem} implies that there exists a non-trivial reducible $p$-ordinary deformation $\rho : G_{\QQ,p\ell} \to \GL_2(\FF_p[\epsilon]/(\epsilon^2))$ of $\rhob_{c_0}$ with determinant $\omega_p^{k-1}$.
So $\tr(\rho) = \chi_1 + \chi_2$, where $\chi_1, \chi_2 : G_{\QQ,p\ell} \to (\FF_p[\epsilon]/(\epsilon^2))^{\times}$ are characters deforming $1$ and $\omega_p^{k-1}$, respectively.
Since $\rho$ is $p$-ordinary, $\chi_1$ is unramified at $p$.
If $g \in I_{\ell}$, then $\chi_1(g) = 1+a_g\epsilon$ for some $a_g \in \FF_p$. As $d(g)=\chi_1(g)\chi_2(g) =\omega_p^{k-1}(g)=1$, so $\chi_2(g) = (1+a_g\epsilon)^{-1}=1 - a_g\epsilon$. So $\tr(\rho(g))=2$ for all $g \in I_{\ell}$.
Therefore, $(\tr(\rho),\det(\rho)) : G_{\QQ,p\ell} \to \FF_p[\epsilon]/(\epsilon^2)$ is a non-trivial, reducible, $p$-ordinary, $\ell$-unipotent pseudo-representation with determinant $\omega_p^{k-1}$ deforming $(\tr(\rhob_0),\det(\rhob_0))$.

From the proof of Lemma~\ref{tanlem}, we know that the space of first order deformations of $(\tr(\rhob_0),\det(\rhob_0))$ which are reducible, $p$-ordinary, $\ell$-unipotent with determinant $\omega_p^{k-1}$ has dimension $1$.
Therefore, we can find a deformation $\rho'$ of $\rhob_{c_0}$ in the subspace of the first order deformations of $\rhob_{c_0}$ generated by $\rho$ such that $\tr(\rho) =t$ and $\det(\rho)=d$. This proves the lemma.
\end{proof}

\section{Hecke algebras}
\label{heckesec}
We will now introduce the Hecke algebras that we will be working with and collect their properties. We will mostly follow \cite{W} in this section.
Let $M_k(\ell,\ZZ_p)$ be the space of classical modular forms of level $\Gamma_0(\ell)$ and weight $k$ with Fourier coefficients in $\ZZ_p$ and let $S_k(\ell,\ZZ_p)$ be its submodule of cusp forms.
Let $\TT$ be the $\ZZ_p$-subalgebra of $\text{End}_{\ZZ_p}(M_k(\ell,\ZZ_p))$ generated by the Hecke operators $T_q$ for primes $q \neq \ell$ and the Atkin--Lehner operator $w_{\ell}$ at $\ell$.
Let $\TT^0$ be the $\ZZ_p$-subalgebra of $\text{End}_{\ZZ_p}(S_k(\ell,\ZZ_p))$ generated by the Hecke operators $T_q$ for primes $q \neq \ell$ and the Atkin--Lehner operator $w_{\ell}$ at $\ell$.

The restriction of the action of Hecke operators from $M_k(\ell,\ZZ_p)$ to $S_k(\ell,\ZZ_p)$ gives a surjective morphism $\TT \to \TT^0$.
Let $I^{\eis}$ be the ideal of $\TT$ generated by the set $\{w_{\ell}+1, T_q - (1+q^{k-1}) \mid q \neq \ell \text{ is a prime}\}$. It is easy to verify that $I^{\eis}$ is a prime ideal $\TT$ and it corresponds to the Eisenstein series of level $\Gamma_0(\ell)$ and weight $k$ having $w_{\ell}$-eigenvalue $-1$ (see \cite[Section $2$]{W} for more details).
Let $\fm$ be the ideal of $\TT$ generated by $p$ and $I^{\eis}$. So $\fm$ is a maximal ideal of $\TT$.
Denote by $\TT_{\fm}$ and $\TT^0_{\fm}$ the completion of $\TT$ and $\TT^0$ at $\fm$, respectively.
It follows from \cite{W} that $\TT^0_{\fm}$ is non-zero.
 Note that the surjective map $\TT \to \TT^0$ induces a surjective map $F : \TT_{\fm} \to \TT^0_{\fm}$.

Let $\fm^0$ be the maximal ideal of $\TT^0$.
As $\TT$ is a finite $\ZZ_p$-module, \cite[Corollary 7.6]{E} implies that $\TT_{\mathfrak{m}}$ is the localization of $\TT$ at $\fm$ and it is a finite $\ZZ_p$-module.
Similarly, $\TT^0_{\fm}$ is the localization of $\TT^0$ at $\fm^0$ and it is a finite $\ZZ_p$-module.
Note that both $\TT$ and $\TT^0$ are reduced.
As $\TT_{\fm}$ and $\TT^0_{\fm}$ are localizations of $\TT$ and $\TT^0$, respectively, we get that both $\TT_{\fm}$ and $\TT^0_{\fm}$ are reduced.

The residue field of both $\TT_{\fm}$ and $\TT^0_{\fm}$ is $\FF_p$.
Denote the maximal ideals of $\TT_{\fm}$ and $\TT^0_{\fm}$ by $\fm$ and $\fm^0$, respectively.
By abuse of notation, denote by $I^{\eis}$ the ideal of $\TT_{\fm}$ generated by the set $\{w_{\ell}+1, T_q - (1+q^{k-1}) \mid q \neq \ell \text{ is a prime}\}$.
We call it the Eisenstein ideal of $\TT_{\fm}$.
Denote by $I^{\eis,0}$ the ideal of $\TT^0_{\fm}$ generated by the set $\{w_{\ell}+1, T_q - (1+q^{k-1}) \mid q \neq \ell \text{ is a prime}\}$.
We call it the Eisenstein ideal of $\TT^0_{\fm}$.

We will now collect some of the properties of $\TT_{\fm}$ and $\TT^0_{\fm}$.
We begin by relating $\TT_{\fm}$ with the pseudo-deformation ring $R^{\ps,\ord}_{\rhob_0,k}(\ell)$ introduced in \S\ref{defsec}.
\begin{lem}
\label{pseudoordlem}
There exists a pseudo-representation $(\tau_{\ell}, \delta_{\ell}) : G_{\QQ,p\ell} \to \TT_{\fm}$ such that $(\tau_{\ell}, \delta_{\ell})$ is a $p$-ordinary, Steinberg-or-unramified at $\ell$ deformation of $(\tr(\rhob_0),\det(\rhob_0))$ with determinant $\chi_p^{k-1}$ and $\tau_{\ell}(\frob_q) = T_q$ for all primes $q \nmid p\ell$. The morphism $\phi_{\TT} : R^{\ps,\ord}_{\rhob_0,k}(\ell) \to \TT_{\fm}$ induced by $(\tau_{\ell},\delta_{\ell})$ is surjective.
\end{lem}
\begin{proof}
For $k>2$, the lemma follows from \cite[Section $3.2$]{W}. For $k=2$, the lemma follows from \cite[Proposition $4.2.4$]{WWE1}.
Indeed, the pseudo-representation $(\tau_{\ell},\delta_{\ell})$ is obtained by gluing the pseudo-representations corresponding to the semi-simple $p$-adic Galois representations attached to the modular eigenforms of level $\Gamma_0(\ell)$ and weight $k$ lifting $\rhob_0$.
Here we say that an eigenform $f$ of level $\Gamma_0(\ell)$ and weight $k$ lifts $\rhob_0$ if $a_q(f) \equiv 1+q^{k-1} \pmod{\varpi_f}$ for all primes $q \neq \ell$, where $a_q(f)$ is the $T_q$-eigenvalue of $f$ and $\varpi_f$ is a uniformizer of the ring of integers of the finite extension of $\QQ_p$ obtained by attaching the Hecke eigenvalues of $f$ to $\QQ_p$.

Suppose $f$ is an eigenform of level $\Gamma_0(\ell)$ and weight $k$ lifting $\rhob_0$ and $\rho_f$ is the semi-simple $p$-adic Galois representation attached to $f$.
Then the $T_p$-eigenvalue of $f$ is a $p$-adic unit and hence, $\rho_f$ is $p$-ordinary. This means that the pseudo-representation attached to $\rho_f$ is also $p$-ordinary (see \S\ref{repsubsec}).
As $f$ is of level $\Gamma_0(\ell)$, we know that $\rho_f|_{G_{\QQ_{\ell}}}$ is either unramified or Steinberg. Hence, the pseudo-representation attached to $\rho_f$ is Steinberg-or-unramified at $\ell$ (see \cite[Observation $1.9.2$]{WWE2}). This proves the desired properties of the pseudo-representation $(\tau_{\ell},\delta_{\ell})$.
The surjectivity of $\phi_{\TT}$ can be concluded in the same way as in \cite[Proposition $4.2.4$]{WWE1}.
\end{proof}

We now prove that the $\TT^0_{\fm}$-valued pseudo-representation obtained by composing $(\tau_{\ell},\delta_{\ell})$ with the natural surjective map $F : \TT_{\fm} \to \TT^0_{\fm}$ arises from an actual representation.

\begin{lem}
\label{ordlem}
There exists a $p$-ordinary deformation $\rho_{\TT^0} : G_{\QQ,p\ell} \to \GL_2(\TT^0_{\fm})$ of $\rhob_{c_0}$ with determinant $\chi_p^{k-1}$ such that $\tr(\rho_{\TT^0}(\frob_q)) = T_q$ for all primes $q \nmid p\ell$ and $\tr(\rho_{\TT^0}(g)) =2$ for all $g \in I_{\ell}$.
The morphism $\phi_{\TT^0} : R^{\defo,\ord}_{\rhob_{c_0},k}(\ell) \to \TT^0_{\fm}$ induced by $\rho_{\TT^0}$ is surjective.
\end{lem}
\begin{proof}
Composing $(\tau_{\ell},\delta_{\ell})$ with the surjective map $F : \TT_{\fm} \to \TT^0_{\fm}$ gives us a pseudo-representation $(\tau^0_{\ell},\delta^0_{\ell}) : G_{\QQ,p\ell} \to \TT^0_{\fm}$ which is a $p$-ordinary, Steinberg-or-unramified at $\ell$ deformation of $(\tr(\rhob_0),\det(\rhob_0))$ with determinant $\chi_p^{k-1}$.
Let $A=\begin{pmatrix} \TT^0_{\fm} & B \\ C & \TT^0_{\fm} \end{pmatrix}$ be the GMA over $\TT^0_{\fm}$ and $\rho : G_{\QQ,p\ell} \to A^{\times}$ the representation associated to $(\tau^0_{\ell},\delta^0_{\ell})$ by Lemma~\ref{gmalem}.

Note that $\TT^0_{\fm}$ is reduced. Let $K_0$ be the total fraction ring of $\TT^0_{\fm}$.
So, by \cite[Proposition $1.3.12$]{BC}, we can assume that $B$ and $C$ are fractional ideals of $K_0$ and the multiplication between $B$ and $C$ is given by the multiplication in $K_0$.
If $I=BC$, then $\tau^0_{\ell} \pmod{I}$ is a sum of two characters. But minimal primes of $\TT^0_{\fm}$ correspond to cuspidal eigenforms $f$ of level $\Gamma_0(\ell)$ and weight $k$ lifting $\rhob_0$. Hence, if $P$ is a minimal prime of $\TT^0_{\fm}$, then $(\tau^0_{\ell} \pmod{P},\delta^0_{\ell}\pmod{P})$ is not reducible as it is the pseudo-representation corresponding to the $p$-adic Galois representation attached to the cuspidal eigenform corresponding to $P$. 
So $I$ is not contained in any minimal prime of $\TT^0_{\fm}$. So $B \neq 0$. 

If $\alpha B =0$, then we have $\alpha I=0$. So $\alpha$ should be in every minimal prime of $\TT^0_{\fm}$. As $\TT^0_{\fm}$ is reduced, it means that $\alpha = 0$. 
Now if $\rho(i_{\ell}) = \begin{pmatrix} 1+x & b_{\ell} \\ c_{\ell} & 1-x\end{pmatrix}$, then part~\eqref{bhaag4} of Lemma~\ref{gmalem} implies that $B$ is generated by $b_{\ell}$ over $\TT^0_{\fm}$. 
Hence, $B$ is a free $\TT^0_{\fm}$-module of rank $1$ generated by $b_{\ell}$ over $\TT^0_{\fm}$.
As $I$ is not contained in any of the minimal primes of $\TT^0_{\fm}$, it follows that $b_{\ell} \in K_0^{\times}$.

So $\begin{pmatrix} b_{\ell}^{-1} & 0\\ 0 & 1\end{pmatrix} A \begin{pmatrix} b_{\ell} & 0\\ 0 & 1\end{pmatrix} = \begin{pmatrix}\TT^0_{\fm} & \TT^0_{\fm}\\ I & \TT^0_{\fm}\end{pmatrix}$.
Hence, conjugating $\rho$ by $\begin{pmatrix} b_{\ell}^{-1} & 0\\ 0 & 1\end{pmatrix}$ gives us a representation $\rho' : G_{\QQ,p\ell} \to \GL_2(\TT^0_{\fm})$.
As $\tr(\rho)=\tau^0_{\ell}$, it follows that $\tr(\rho')=\tau^0_{\ell}$. Hence, we get $\tr(\rho'(g))=2$ for all $g \in I_{\ell}$.
From part~\eqref{bhaag6} of Lemma~\ref{gmalem}, it follows that $\rho(h) = \begin{pmatrix} 1 & 0\\ c_h & \chi_p^{k-1}(h)\end{pmatrix}$ for all $h \in I_p$.
Hence, $\rho'$ is $p$-ordinary as it is a conjugate of $\rho$ by a diagonal matrix.

Therefore, $\rho' \pmod{\fm^0} = \begin{pmatrix}1 & *\\ 0 & \omega_p^{k-1}\end{pmatrix}$ where $*$ is non-zero and unramified at $p$ i.e. $\rho'\pmod{\fm^0}$ arises from an element of $\ker(H^1(G_{\QQ,p\ell},\omega_p^{1-k}) \to H^1(G_{\QQ_p},\omega_p^{1-k}))$.
By Lemma~\ref{cohomlem}, it follows that $\ker(H^1(G_{\QQ,p\ell},\omega_p^{1-k}) \to H^1(G_{\QQ_p},\omega_p^{1-k}))$ is generated by $c_0$.
Hence, conjugating $\rho'$ with a suitable diagonal matrix gives us the representation $\rho_{\TT^0} : G_{\QQ,p\ell} \to \GL_2(\TT^0_{\fm})$ satisfying the statement of the lemma.

The existence of $\rho_{\TT^0}$ implies that the map $R^{\ps,\ord}_{\rhob_0,k}(\ell) \to \TT^0_{\fm}$ induced by $(\tau^0_{\ell},\delta^0_{\ell})$ factors through $R^{\defo,\ord}_{\rhob_{c_0},k}(\ell)$ to give the map $\phi_{\TT^0}$ induced by $\rho_{\TT^0}$. Hence the surjection of $\phi_{\TT^0}$ follows from Lemma~\ref{pseudoordlem}.
\end{proof}

We now show that the space of first order deformations of $(\tr(\rhob_0),\det(\rhob_0))$ arising from $\TT_{\fm}$ always contains reducible deformations.

\begin{lem}
\label{principallem}
If $\dim(\tan(\TT_{\fm}/(p))) = 1$, then $I^{\eis}$ is principal and the pseudo-representation $(\tau_{\ell} \pmod{(p,\fm^2)},\delta_{\ell} \pmod{(p,\fm^2)}) : G_{\QQ,p\ell} \to \TT_{\fm}/(p,\fm^2)$ is reducible.
\end{lem}
\begin{proof}
If $\dim(\tan(\TT_{\fm}/(p))) = 1$, then $\TT_{\fm}$ is a quotient of $\ZZ_p\llbracket X \rrbracket$. As $\TT_{\fm}/I^{\eis} \simeq \ZZ_p$, it follows that $I^{\eis}$ is principal.
For $k > 2$, the reducibility of $(\tau_{\ell} \pmod{(p,\fm^2)},\delta_{\ell} \pmod{(p,\fm^2)})$ follows from \cite[Theorem $5.1.1$]{W}.

Suppose $k=2$. Then the lemma follows from work of Calegari and Emerton (by combining \cite[Proposition 3.12]{CE} and \cite[Proposition 5.5]{CE}).
However, we will give a different proof here as we are not using the deformation conditions studied by them.
It follows from \cite{M2} that $\dim(\tan(\TT_{\fm}/(p))) = 1$ and $I^{\eis}$ is principal.
So $\TT_{\fm}/(p,\fm^2) \simeq \FF_p[\epsilon]/(\epsilon^2)$.
Denote $\TT_{\fm}/(p,\fm^2)$ by $R$ and $(\tau_{\ell} \pmod{(p,\fm^2)},\delta_{\ell} \pmod{(p,\fm^2)})$ by $(t,d)$.
Suppose $(t,d)$ is not reducible.

Let $A$ be the faithful GMA over $R$ and $\rho : G_{\QQ,p\ell} \to A^{\times}$ be the representation associated to $(t,d)$ by Lemma~\ref{gmalem}.
From the proof of Lemma~\ref{cuplem}, it follows that there exist non-zero constants $\alpha, \beta, \gamma \in \FF_p$ such that if $\rho(g) = \begin{pmatrix} a_g & b_g \\ c_g & d_g \end{pmatrix}$, then $b_gc_g=\alpha\beta\gamma\omega_p^{k-1}(g)c_0(g)b_0(g)\epsilon$ for all $g \in G_{\QQ,p\ell}$.
Let $K_{c_0}$ be the extension of $\QQ$ fixed by the kernel of the representation $\begin{pmatrix} \omega_p^{1-k} & *\\ 0 & 1\end{pmatrix}$ defined by $c_0$ and $K_{b_0}$ be the extension of $\QQ$ fixed by the kernel of the representation $\begin{pmatrix} \omega_p^{k-1} & *\\ 0 & 1\end{pmatrix}$ defined by $b_0$.

By Chebotarev density theorem, there exists a prime $q$ such that $q \nmid p\ell$, $p \mid q-1$ and $q$ is not totally split in both $K_{c_0}$ and $K_{b_0}$.
This means that $c_0(\frob_q) \neq 0$ and $b_0(\frob_q) \neq 0$.
So if $\rho(\frob_q) = \begin{pmatrix} a & b\\ c & d\end{pmatrix}$, then $bc \neq 0$. Now $a = 1+x\epsilon$, $d = 1 +y\epsilon$ with $x,y \in \FF_p$ and $\det(\rho(\frob_q))=1$.
Hence, it follows that $x+y \neq 0$. Therefore, $\tr(\rho(\frob_q)) - q-1 \neq 0$ and hence, it generates the cotangent space of $R$. 
The image of $T_q$ under the surjective map $\TT_{\fm} \to R$ is $\tr(\rho(\frob_q))$.
Hence, it follows that $p$ and $T_q-q-1$ generate the maximal ideal $\fm$ of $\TT_{\fm}$.
 Note that $T_q-q-1 \in I^{\eis}$ and $\TT_{\fm}/(p, T_q-q-1) \simeq \FF_p$.
As $\TT_{\fm}/I^{\eis} \simeq \ZZ_p$, we get that $T_q-q-1$ generates $I^{\eis}$.
Since $q$ is not a nice prime in the sense of Mazur (\cite{M2}), \cite[Proposition II.$16.1$]{M2} gives a contradiction.
Hence, $(t,d)$ is reducible which implies the lemma.
\end{proof}

We will now briefly review modular forms modulo $p$ as they form a crucial ingredient of the proof of Theorem~\ref{thm3} and Corollary~\ref{main}.

Let $i > 0$ be an even integer and $M_i(\ell,\ZZ_p)$ be the space of classical modular forms of level $\Gamma_0(\ell)$ and weight $i$ with Fourier coefficients in $\ZZ_p$. Using the $q$-expansion principle, we identify $M_i(\ell,\ZZ_p)$ with a submodule of $\ZZ_p \llbracket q \rrbracket$.
Let $M_i(\ell,\FF_p)$ be the image of $M_i(\ell,\ZZ_p)$ under the natural surjective map $\ZZ_p \llbracket q \rrbracket \to \FF_p \llbracket q \rrbracket$ obtained by reducing the coefficients of power series modulo $p$. So $M_i(\ell,\FF_p)$ is the space of modular forms modulo $p$ of weight $i$ and level $\Gamma_0(\ell)$ (in the sense of Serre and Swinnerton-Dyer).

Let $\TT_i$ be the $\ZZ_p$-subalgebra of $\text{End}_{\ZZ_p}(M_i(\ell,\ZZ_p))$ generated by the Hecke operators $T_q$ for primes $q \neq \ell$ and the Atkin--Lehner operator $w_{\ell}$ at $\ell$.
So, under the notation developed above, we have $\TT_k = \TT$.
Let $\mathfrak{n}$ be a maximal ideal of $\TT_i$ and let $(\TT_i)_{\mathfrak{n}}$ be the completion of $\TT_i$ at $\mathfrak{n}$.
As $\TT_i$ is a finite $\ZZ_p$-module, \cite[Corollary 7.6]{E} implies that $(\TT_i)_{\mathfrak{n}}$ is the localization of $\TT_i$ at $\mathfrak{n}$ and $\TT_i = (\TT_i)_{\mathfrak{n}} \times S$, where $S$ is the product of localizations of $\TT_i$ at maximal ideals other than $\mathfrak{n}$.

Let $M_i(\ell,\ZZ_p)_{\mathfrak{n}}$ be the localization of $M_i(\ell,\ZZ_p)$ at $\mathfrak{n}$.
From the product decomposition of $\TT_i$ given in the previous paragraph, we conclude that $M_i(\ell,\ZZ_p)_{\mathfrak{n}}$ is a submodule of $M_i(\ell,\ZZ_p)$ and moreover, it is a direct summand of $M_i(\ell,\ZZ_p)$.
Note that $(\TT_i)_{\mathfrak{n}}$ is the largest quotient of $\TT_i$ acting faithfully on $M_i(\ell,\ZZ_p)_{\mathfrak{n}}$.

The action of $\TT_i$ on $M_i(\ell,\ZZ_p)$ also gives an action of $\TT_i$ on $M_i(\ell,\FF_p)$ and this action factors through $\TT_i/(p)$.
Let $M_i(\ell,\FF_p)_{\mathfrak{n}}$ be the localization of $M_i(\ell,\FF_p)$ at the maximal ideal $\mathfrak{n}$. 
As $\TT_i/(p)$ is Artinian, it follows that $M_i(\ell,\FF_p)_{\mathfrak{n}}$ is the submodule of $M_i(\ell,\FF_p)$ consisting of generalized eigenvectors corresponding to the system of eigenvalues defined by $\mathfrak{n}$. In other words, $M_i(\ell,\FF_p)_{\mathfrak{n}} = \{ f \in M_i(\ell,\FF_p) \mid \mathfrak{n}^k.f=0 \text{ for some } k > 0 \}$.

Under the mod $p$ reduction map $M_i(\ell,\ZZ_p) \to M_i(\ell,\FF_p)$, $M_i(\ell,\ZZ_p)_{\mathfrak{n}}$ gets mapped onto $M_i(\ell,\FF_p)_{\mathfrak{n}}$.
Thus the action of $\TT_i$ on $M_i(\ell,\FF_p)_{\mathfrak{n}}$ factors through $(\TT_i)_{\mathfrak{n}}/(p)$.
\begin{lem}
\label{quotient}
If $(p-1) \nmid i$, then $(\TT_i)_{\mathfrak{n}}/(p)$ is the largest quotient of $(\TT_i)_{\mathfrak{n}}$ acting faithfully on $M_i(\ell,\FF_p)_{\mathfrak{n}}$.
\end{lem}

\begin{proof}
As $p$ is odd and $w_{\ell}$ is an involution, it follows that the image of $w_{\ell}$ in $(\TT_i)_{\mathfrak{n}}$ is either $1$ or $-1$.
Since $i > 0$, we get, from \cite[Corollary 2.1.4]{Oh}, a perfect pairing $$G: (\TT_i)_{\mathfrak{n}} \times M_i(\ell,\ZZ_p)_{\mathfrak{n}} \to \ZZ_p$$ which sends $(T,f)$ to $a_1(Tf)$, where $a_1(Tf)$ is the coefficient of $q$ in the $q$-expansion of $Tf$.

Let $\overline{\mathbb{T}_i}$ be the largest quotient of $(\TT_i)_{\mathfrak{n}}$ acting faithfully on $M_i(\ell,\FF_p)_{\mathfrak{n}}$.
As $(p-1) \nmid i$, we know that no non-zero modular form in $M_i(\ell,\FF_p)$ has \emph{constant} Fourier expansion i.e. $M_i(\ell,\FF_p) \setminus \{0\} \subset \FF_p \llbracket q \rrbracket \setminus \FF_p$ (see the discussion on Page $459$ of \cite{G} for more details).
Hence, by applying \cite[Corollary 2.1.4]{Oh} again, we get that the map $$\bar{G}: \overline{\TT_i} \times M_i(\ell,\FF_p)_{\mathfrak{n}} \to \FF_p$$ which sends $(T,\bar{f})$ to $a_1(T\bar{f})$ is a perfect pairing.

As $G$ is a perfect pairing, we get that the $\ZZ_p$-ranks of $(\TT_i)_{\mathfrak{n}}$ and $M_i(\ell,\ZZ_p)_{\mathfrak{n}}$ are equal.
Recall that $M_i(\ell,\FF_p)_{\mathfrak{n}}$ is the reduction of $M_i(\ell,\ZZ_p)_{\mathfrak{n}}$ modulo $p$.
So, we conclude that the $\FF_p$-dimension of $M_i(\ell,\FF_p)_{\mathfrak{n}}$ is same as the $\ZZ_p$-rank of $M_i(\ell,\ZZ_p)_{\mathfrak{n}}$.
Since $\bar{G}$ is a perfect pairing, the $\FF_p$-dimensions of $\overline{\TT_i}$ and $M_i(\ell,\FF_p)_{\mathfrak{n}}$ are the same.
Therefore, the $\FF_p$-dimensions of $\overline{\TT_i}$  and $(\TT_i)_{\mathfrak{n}}/(p)$ are equal.
As $(\TT_i)_{\mathfrak{n}}/(p)$ surjects onto $\overline{\TT_i}$, we infer that $(\TT_i)_{\mathfrak{n}}/(p) \simeq \overline{\TT_i}$.
\end{proof}

We will now relate the principality of $I^{\eis,0}$ with that of $I^{\eis}$. This result will be used in the proof of Theorem~\ref{thm2}.
Recall that $F : \TT_{\fm} \to \TT^0_{\fm}$ is the map induced from the natural surjective map $\TT \to \TT^0$.
\begin{lem}
\label{principal}
$I^{\eis}$ is principal if and only if $I^{\eis,0}$ is principal.
\end{lem}
\begin{proof}
Since $I^{\eis,0}$ is the ideal generated by $F(I^{\eis})$, it follows that $I^{\eis,0}$ is principal if $I^{\eis}$ is principal.

Now suppose $I^{\eis,0}$ is principal.
From Lemma~\ref{tanlem} and Lemma~\ref{pseudoordlem}, we know that there exists a surjective map $f : \ZZ_p \llbracket x,y \rrbracket \to \TT_{\fm}$.
Moreover, we can choose this map so that $f(x), f(y) \in I^{\eis}$ and hence, $I^{\eis} = (f(x),f(y))$.
As $I^{\eis,0}$ is principal and $F$ is surjective, we get, using Nakayama's lemma, that $I^{\eis}$ is either $(F(f(x)))$ or $(F(f(y)))$.
Hence, $\ker(F)$ contains either $f(y)-rf(x)$ for some $r \in \TT_{\fm}$ or $f(x)-r'f(y)$ for some $r' \in \TT_{\fm}$.

Suppose $f(y)-rf(x) \in \ker(F)$ for some $r \in \TT_{\fm}$.
Recall that $f(y)-rf(x) \in I^{\eis}$.
Suppose the $\ZZ_p$-rank of $M_k(\ell,\ZZ_p)_{\fm}$ is $d$.
Note that, the Eisenstein subspace $E_k(\ell,\ZZ_p)_{\fm}$ of $M_k(\ell,\ZZ_p)_{\fm}$ has $\ZZ_p$-rank $1$ and the cuspidal subspace $S_k(\ell,\ZZ_p)_{\fm}$ of $M_k(\ell,\ZZ_p)_{\fm}$ has $\ZZ_p$-rank $d-1$ (see \cite[Section 2.2]{W} for more details).

Now $E_k(\ell,\ZZ_p)_{\fm} \cap S_k(\ell,\ZZ_p)_{\fm} = \{0\}$.
Hence, if $g \in M_k(\ell,\ZZ_p)_{\fm}$, then there exist $g' \in E_k(\ell,\ZZ_p)_{\fm}$, $h \in S_k(\ell,\ZZ_p)_{\fm}$ and an integer $n \geq 0$ such that $g =\dfrac{g'-h}{p^n}$.
Thus if $\sigma \in I^{\eis}$, then $\sigma(g')=0$, $\sigma(h) \in S_k(\ell,\ZZ_p)_{\fm}$ and $\sigma(g) \in M_k(\ell,\ZZ_p)_{\fm}$.
Therefore, we conclude that $\sigma(g) \in S_k(\ell,\ZZ_p)_{\fm}$ and hence, $\sigma.M_k(\ell,\ZZ_p)_{\fm} \subset S_k(\ell,\ZZ_p)_{\fm}$.

Since $f(y)-rf(x) \in I^{\eis} \cap \ker(F)$, $(f(y)-rf(x))^2.M_k(\ell,\ZZ_p)_{\fm}=0$ i.e. $(f(y)-rf(x))^2=0$.
Recall that $\TT_{\fm}$ is reduced.
Therefore, we get that $f(y)-rf(x)=0$. As $I^{\eis}=(f(x),f(y))$, we conclude that $I^{\eis}$ is principal.

Using the same argument as above, we get that if $f(x)-r'f(y) \in \ker(F)$ for some $r' \in \TT_{\fm}$, then $I^{\eis}$ is a principal ideal.
This finishes the proof of the lemma.
\end{proof}

We will now consider $I^{\eis}$ as a $\TT_{\fm}$-module and determine its annihilator.
This result will be crucially used in the proofs of Theorem~\ref{thm2} and Part~\eqref{item1} of Theorem~\ref{thmb}.

\begin{lem}
\label{annih}
The annihilator of the $\TT_{\fm}$-module $I^{\eis}$ is $\ker(F)$.
\end{lem}
\begin{proof}
Suppose $\alpha \in \ker(F)$ and $\beta \in I^{\eis}$.
Recall, from the proof of Lemma~\ref{principal}, that $\beta M_k(\ell,\ZZ_p)_{\fm} \subset S_k(\ell,\ZZ_p)_{\fm}$.
So $\alpha\beta M_k(\ell,\ZZ_p)_{\fm} =0$.
Therefore, it follows that $\alpha\beta=0$ for all $\beta \in I^{\eis}$ and $\alpha \in \ker(F)$.

On the other hand, suppose $\alpha \in \TT_{\fm}$ and $\alpha I^{\eis}=0$.
As $\TT_{\fm}$ is reduced, $\alpha \not\in I^{\eis}$.
Since $\TT_{\fm}$ has Krull dimension $1$ and $\TT_{\fm}/I^{\eis} \simeq \ZZ_p$, $I^{\eis}$ is a minimal prime ideal of $\TT_{\fm}$.
Let $\mathcal{S}$ be the set of all minimal prime ideals of $\TT_{\fm}$ which are different from $I^{\eis}$.
As $\TT^0_{\fm}$ is non-zero, $\mathcal{S}$ is non-empty.
So, if $P \in \mathcal{S}$, then $I^{\eis} \not\subset P$ and hence, $\alpha \in P$. Thus, $\alpha \in \cap_{P \in \mathcal{S}} P$.

Now we have $I^{\eis}\ker(F)=0$.
Therefore, if $P \in \mathcal{S}$, then $\ker(F) \subset P$.
As $\TT_{\fm}$ is a local ring of Krull dimension $1$, $\mathcal{S}$ is the set of all primes of $\TT_{\fm}$ which are minimal over $\ker(F)$.
Since $\TT^0_{\fm} \simeq \TT_{\fm}/\ker(F)$ is reduced, we conclude that $\cap_{P \in \mathcal{S}} P =\ker(F)$.
Therefore, we get that $\alpha \in \ker(F)$.
\end{proof}

\section{Main results}
\label{mainsec}
We are now ready to prove our main results. We know, from \cite{M2} and \cite{W}, that $\dim(\tan(\TT_{\fm}/(p))) \geq 1$.

\begin{thm}
\label{thm1}
Suppose $\dim(\tan(\TT_{\fm}/(p))) = 1$. Then $\rank_{\ZZ_p}(\TT^0_{\fm}) = 1$ if and only if $c_0 \cup a_0 \neq 0$.
\end{thm}
\begin{proof}
As $\dim(\tan(\TT_{\fm}/(p))) = 1$, Lemma~\ref{principallem} implies that $I^{\eis}$ is principal. Let $x_0 \in \TT_{\fm}$ be a generator of $I^{\eis}$.
Then $\fm = (p,x_0)$ and $(p,\fm^2)=(p,x_0^2)$. So $\TT_{\fm}/(p,\fm^2) = \TT_{\fm}/(p,x_0^2) \simeq \FF_p[\epsilon]/(\epsilon^2)$.
Let $f : \TT_{\fm} \to \FF_p[\epsilon]/(\epsilon^2)$ be the map obtained by composing the isomorphism obtained above with the natural surjective map $\TT_{\fm} \to \TT_{\fm}/(p,\fm^2)$.
Now Lemma~\ref{principallem} also implies that the pseudo-representation $(t,d) : G_{\QQ,p\ell} \to \FF_p[\epsilon]/(\epsilon^2)$ obtained by composing $(\tau_{\ell},\delta_{\ell})$ with the surjective map $f : \TT_{\fm} \to \FF_p[\epsilon]/(\epsilon^2)$ is reducible.

If $\rank_{\ZZ_p}(\TT^0_{\fm}) > 1$, then $\dim(\tan(\TT^0_{\fm}/(p))) \geq 1$.
As $\dim(\tan(\TT_{\fm}/(p))) = 1$, we get that $\dim(\tan(\TT^0_{\fm}/(p))) = 1$.
Therefore, the map $f$ factors through $\TT^0_{\fm}$. Thus Lemma~\ref{ordlem} implies that $(t,d)$ arises from a non-trivial first order $p$-ordinary deformation of $\rhob_{c_0}$ with determinant $\omega_p^{k-1}$. Hence, Lemma~\ref{inflem} implies that $c_0 \cup a_0=0$.

Now suppose $c_0 \cup a_0 =0$. Recall that $\phi : R^{\ps,\ord}_{\rhob_0,k}(\ell) \to R^{\defo,\ord}_{\rhob_{c_0},k}(\ell)$ is the surjective morphism induced by $(\tr(\rho^{\univ,\ell}), \det(\rho^{\univ,\ell}))$.
As $(t,d)$ is reducible, Lemma~\ref{infinlem} implies that there exists a map $f' : R^{\defo,\ord}_{\rhob_{c_0},k}(\ell) \to \FF_p[\epsilon]/(\epsilon^2)$ such that the following diagram commutes:
\[
\begin{tikzcd} 
R^{\ps,\ord}_{\rhob_0,k}(\ell)\arrow{r}{\phi_{\TT}} \arrow{dr}{\phi} & \TT_{\fm} \arrow{r}{f} & \FF_p[\epsilon]/(\epsilon^2)\\
& R^{\defo,\ord}_{\rhob_{c_0},k}(\ell) \arrow{ur}{f'}
\end{tikzcd}
\]
Hence, $\phi_{\TT}(\ker(\phi)) \subset (p,x_0^2)$.

On the other hand, Lemma~\ref{ordlem} implies that the following diagram commutes:
\[
\begin{tikzcd} 
R^{\ps,\ord}_{\rhob_0,k}(\ell)\arrow{r}{\phi_{\TT}} \arrow{dr}{\phi} & \TT_{\fm} \arrow{r}{F} & \TT^0_{\fm} \\
& R^{\defo,\ord}_{\rhob_{c_0},k}(\ell) \arrow{ur}{\phi_{\TT^0}}
\end{tikzcd}
\]
So $\phi_{\TT}(\ker(\phi)) \subset \ker(F)$.

Now suppose $\rank_{\ZZ_p}(\TT^0_{\fm}) = 1$. As $I^{\eis}=(x_0)$, \cite[Theorem $5.1.2$]{W} and \cite[Proposition II.$9.6$]{M2} imply that $F(x_0)=p^{\nu+v_p(k)}.u$ for some $u \in \ZZ_p^{\times}$ (see \cite[Remark $5.1.3$]{W} for more details).
Since the image of $x_0$ in $\TT_{\fm}/(p)$ generates its cotangent space, it follows that $\ker(F) = (x_0-p^{\nu+v_p(k)}.u)$.

Note that $J_0$, which is the Eisenstein ideal of $R^{\ps,\ord}_{\rhob_0,k}(\ell)$, is the inverse image of $I^{\eis}$ under the surjective map $\phi_{\TT}$. 
Then $I_0$ is generated by $\phi(J_0)$ and Lemma~\ref{eisenlem} implies that there is a $y \in \ker(\phi)$ such that $y = p^e +y_0$ with $y_0 \in J_0$ and $e \leq \nu+v_p(k)$.
So $\phi_{\TT}(y) = p^e + \phi_{\TT}(y_0) \in \ker(F)$. As $y_0 \in J_0$, $\phi_{\TT}(y_0) \in (x_0)$ and hence, $\phi_{\TT}(y_0)=x_0\alpha$ for some $\alpha \in \TT_{\fm}$.
So $p^e + x_0\alpha  \in (x_0-p^{\nu+v_p(k)}.u)$.
As $\TT_{\fm}/(x_0) \simeq \ZZ_p$, it follows that $p^e+x_0\alpha = (x_0-p^{\nu+v_p(k)}.u)(x_0\beta + p^{e'}u')$ for some $e' \geq 0$, $u' \in \ZZ_p^{\times}$ and $\beta \in \TT_{\fm}$ and hence, $p^e = -uu'p^{\nu+v_p(k)+e'}$.

Since $e \leq \nu + v_p(k)$, we get that $e'=0$ and hence, $p^{e'}u'+x_0\beta \in \TT_{\fm}^{\times}$.
Therefore, we conclude that $\fm = (p, p^e+x_0\alpha)$.
But $p^e+x_0\alpha \in \phi_{\TT}(\ker(\phi)) \subset (p,x_0^2)$. Hence, we see that $p$ generates the maximal ideal of $\TT_{\fm}$ which is a contradiction as $\dim(\tan(\TT_{\fm}/(p)))=1$.
Therefore, we conclude that if $c_0 \cup a_0=0$, then $\rank_{\ZZ_p}(\TT^0_{\fm}) > 1$. This finishes the proof of the theorem.
\end{proof}

\begin{cor}
\label{cor1}
If $k=2$, then $\rank_{\ZZ_p}(\TT^0_{\fm}) = 1$ if and only if $c_0 \cup a_0 \neq 0$.
\end{cor}
\begin{proof}
If $k=2$, then \cite[Proposition II.$16.6$]{M2} implies that $\dim(\tan(\TT_{\fm}/(p))) = 1$. So the corollary follows from Theorem~\ref{thm1}.
\end{proof}

\begin{cor}
\label{cor2}
If $k>2$ and $c_0 \cup b_0 \neq 0$, then $\rank_{\ZZ_p}(\TT^0_{\fm}) = 1$ if and only if $c_0 \cup a_0 \neq 0$.
\end{cor}
\begin{proof}
If $k >2$ and $c_0 \cup b_0 \neq 0$, then Lemma~\ref{tanlem}, Lemma~\ref{cuplem} and Lemma~\ref{pseudoordlem} together imply that $\dim(\tan(\TT_{\fm}/(p))) = 1$. Hence, the corollary follows from Theorem~\ref{thm1}.
\end{proof}

Let $\xi'_{\text{MT}} \in \FF_p$ be the derivative of the Mazur--Tate $\zeta$ function defined by Wake in \cite[Section $1.2.2$]{W}.
\begin{cor}
\label{derivative}
If $c_0 \cup b_0 \neq 0$, then $\xi'_{\text{MT}} \neq 0$ if and only if $c_0 \cup a_0 \neq 0$.
\end{cor}
\begin{proof}
If $c_o \cup b_0 \neq 0$, then we know that $I^{\eis}$ is principal. Hence, the corollary follows from Corollary~\ref{cor2} and \cite[Theorem $1.2.4$]{W}.
\end{proof}
Note that the statement of Corollary~\ref{derivative} is purely elementary but its proof is not elementary.
We are not aware of any direct proof of this corollary.

We now move on to the case of $c_0 \cup b_0 =0$.
\begin{thm}
\label{thm2}
If $c_0 \cup b_0 = 0$ and $p \mid k$, then $I^{\eis,0}$ is not principal and $\rank_{\ZZ_p}(\TT^0_{\fm}) > 1$.
\end{thm}
\begin{proof}
If $\rank_{\ZZ_p}(\TT^0_{\fm}) = 1$, then every ideal of $\TT^0_{\fm}$ is principal and hence, $I^{\eis,0}$ is principal.
So it suffices to prove that $I^{\eis,0}$ is not principal to prove the theorem.
Suppose $c_0 \cup b_0=0$, $p \mid k$ and $I^{\eis,0}$ is principal.
By Lemma~\ref{principal}, we get that $I^{\eis}$ is principal.
Since $\TT_{\fm}/I^{\eis} \simeq \ZZ_p$, it follows that $\TT_{\fm}$ is a quotient of $\ZZ_p \llbracket X \rrbracket$ and hence, $\dim(\tan(\TT_{\fm}/(p)))=1$.

Let $A=\begin{pmatrix} \TT_{\fm} & B \\ C & \TT_{\fm}\end{pmatrix}$ be the faithful GMA over $\TT_{\fm}$ and $\rho : G_{\QQ,p\ell} \to A^{\times}$ be the representation attached to $(\tau_{\ell},\delta_{\ell})$ by Lemma~\ref{gmalem}.
By Lemma~\ref{principallem}, we know that the first order deformation of $(\tr(\rhob_0),\det(\rhob_0))$ arising from $\TT_{\fm}$ is reducible.
So Lemma~\ref{tangenlem} implies that $\fm$ is generated by $p$ and $x$, where $\rho(i_{\ell}) =  \begin{pmatrix} 1+x & b_{\ell} \\ c_{\ell} & 1-x\end{pmatrix}$.
As $\det(\rho(i_{\ell}))=1$, we get that $b_{\ell}c_{\ell} = -x^2$.

Let $g_{\ell}$ be a lift of $\frob_{\ell}$ in $G_{\QQ_{\ell}}$ and suppose $\rho(g_{\ell}) = \begin{pmatrix} a & b \\ c & d\end{pmatrix}$.
Now $\rho(g_{\ell}i_{\ell}g_{\ell}^{-1})=\rho(i_{\ell})^{\ell}$.
As $b_{\ell}c_{\ell} = -x^2$, it follows that $\rho(i_{\ell})^{\ell} = \begin{pmatrix} 1+\ell x & \ell b_{\ell} \\ \ell c_{\ell} & 1-\ell x\end{pmatrix} $.
So we have $$\begin{pmatrix} a & b \\ c & d\end{pmatrix}\begin{pmatrix} 1+x & b_{\ell} \\ c_{\ell} & 1-x\end{pmatrix} = \begin{pmatrix} 1+\ell x & \ell b_{\ell} \\ \ell c_{\ell} & 1-\ell x\end{pmatrix}  \begin{pmatrix} a & b \\ c & d\end{pmatrix}.$$
Thus $b_{\ell}c+d(1-x) =\ell bc_{\ell}+ d(1-\ell x)$.
Now part~\eqref{bhaag4} of Lemma~\ref{gmalem} implies that $B=\TT_{\fm}b_{\ell}$ and so $b=rb_{\ell}$ for some $r \in \TT_{\fm}$.
Since $b_{\ell}c_{\ell}=-x^2$, there exists an $r' \in \TT_{\fm}$ such that \begin{equation}\label{eq:1} b_{\ell}c = dx(1-\ell)+x^2r'.\end{equation}

Let $I=BC$.
So, by Lemma~\ref{redlem}, we know that $\tr(\rho)\pmod{I}$ is reducible. 
Moreover, if $\rho(g) = \begin{pmatrix} a_g & b_g \\ c_g & d_g\end{pmatrix}$ for $g \in G_{\QQ,p\ell}$, then the map $\chi_1 : G_{\QQ,p\ell} \to (\TT_{\fm}/I)^{\times}$ sending $g$ to $a_g \pmod{I}$ is a character of $G_{\QQ,p\ell}$ lifting the trivial character. 
Recall, from Part~\eqref{bhaag4} of Lemma~\ref{gmalem}, that the representation $\rho$ is tamely ramified at $\ell$.
So, the character $\chi_1$ is also tamely ramified at $\ell$.
Therefore, by the Kronecker-Weber theorem, we get that the order of $\chi_1(I_{\ell})$ divides $\ell-1$.
Hence, $\chi_1(i_{\ell})^{\ell-1}=1$ which means $(1+x)^{\ell-1} \pmod{I}=1$ i.e. $(1+x)^{\ell-1}-1 \in I$.

We know that $x^2 \in I$.
Since $(1+x)^{\ell-1}-1 \in I$, we get that $p^{\nu}x \in I$. On the other hand, Lemma~\ref{redlem} and \cite[Theorem $5.1.1$]{W} imply that $\TT_{\fm}/I \simeq \dfrac{\ZZ_p \llbracket X \rrbracket}{(X^2,p^{\nu}X)}$. Since we have already seen that $\fm = (p,x)$, there is a surjective morphism $\ZZ_p \llbracket X\rrbracket \to \TT_{\fm}$ sending $X$ to $x$. Hence, combining all this, we get that $I=(x^2,p^{\nu}x)$.

Now as we are assuming $c_0 \cup b_0 = 0$, the proof of Lemma~\ref{vanishlem} implies that $b_0|_{G_{\QQ_{\ell}}} = 0$.

Let $C'=C/\TT_{\fm}c_{\ell}$. Then, following the proof of part~\eqref{bhaag4} of Lemma~\ref{gmalem}, we get an injective map $\psi : \Hom(C'/\fm C', \FF_p) \to H^1(G_{\QQ,p},\omega_p^{k-1})$.
So if its image is non-zero, then it is generated by $b_0$. From the construction of $\psi$ along with the fact $b_0|_{G_{\QQ_{\ell}}} =0$, we see that the image of $c$ in $\fm C'$ is $0$. So $c \in \fm C + \TT_{\fm} c_{\ell}$. Therefore, $b_{\ell}c \in (x^2,p^{\nu+1}x)$ as $\fm =(p,x)$ and $BC=(x^2,p^{\nu}x)$.

Now $d \in \TT_{\fm}^{\times}$.
Hence, from \eqref{eq:1}, we get $x(p^{\nu}+p^{\nu+1}z + xr'')=0$ for some $z \in \ZZ_p$ and $r'' \in \TT_{\fm}$.
From Lemma~\ref{annih}, we know that the annihilator of $I^{\eis}$ is $\ker(F)$.
As $(x)=I^{\eis}$, it follows that $p^{\nu}+p^{\nu+1}z + xr'' \in \ker(F)$. 

So, $|\TT^0_{\fm}/(F(x))| \leq p^{\nu}$.
Since $(x)=I^{\eis}$, it follows that $(F(x)) = I^{\eis,0}$. 
We know, from \cite[Theorem $5.1.2$]{W}, that $\TT^0_{\fm}/F(x) = \TT^0_{\fm}/I^{\eis,0} \simeq \ZZ/p^{\nu+v_p(k)}\ZZ$.
Since $p \mid k$, $v_p(k) >0$ and hence, this gives us a contradiction.
Therefore, we get $\rank_{\ZZ_p}(\TT^0_{\fm}) > 1$ which proves the theorem.
\end{proof}

We will now prove Theorem~\ref{thm2} without the assumption that $p \mid k$. 
We will crucially use the theory of modular forms modulo $p$ (recalled in \S\ref{heckesec}) along with Theorem~\ref{thm2} in its proof.
Note that if $i \geq 2$, then the action of $T_p$ on $M_i(\ell,\FF_p)$ coincides with the action of the operator $U$ considered in \cite[Section $1$]{J} (note that the prime $\ell$ of \cite{J} corresponds to the prime $p$ in our context).

\begin{thm}
\label{thm3}
If $c_0 \cup b_0 = 0$ and $k > 2$, then $I^{\eis,0}$ is not principal and $\rank_{\ZZ_p}(\TT^0_{\fm}) > 1$.
\end{thm}
\begin{proof}
We have already proved the theorem for $p \mid k$. So assume $p \nmid k$ and $c_0 \cup b_0=0$.
Recall that it suffices to prove that $I^{\eis,0}$ is not principal.

Now let $k'$ be an integer such that $k' > k$, $p-1 \mid (k'-k)$ and $p \mid k'$. 
Let $\TT'$ be the $\ZZ_p$-subalgebra of $\text{End}_{\ZZ_p}(M_{k'}(\ell,\ZZ_p))$ generated by the Hecke operators $T_q$ for primes $q \neq \ell$ and the Atkin--Lehner operator $w_{\ell}$ at $\ell$.
Let $\TT'_{\fm}$ be the completion of $\TT'$ at its maximal ideal generated by the set $\{p, w_{\ell}+1, T_q - (1+q^{k'-1}) \mid q \neq \ell \text{ is a prime}\}$.

Let $\mathcal{E}_{p-1}$ be the Eisenstein series of level $1$ and weight $p-1$ such that the constant term of the $q$-expansion $\mathcal{E}_{p-1}(q)$ of $\mathcal{E}_{p-1}$ is $1$.
Note that $\mathcal{E}_{p-1}(q) \in \ZZ_p \llbracket q \rrbracket$.
Let $\overline{\mathcal{E}_{p-1}(q)} \in M_{p-1}(\ell,\FF_p)$ be the reduction of $\mathcal{E}_{p-1}(q)$ modulo $p$.
Then we know that $\overline{\mathcal{E}_{p-1}(q)} = 1$.
Therefore, by using the multiplication by $(\overline{\mathcal{E}_{p-1}(q)})^{\frac{k'-k}{p-1}}$ map, we can identify $M_k(\ell,\FF_p)$ as a subspace of $M_{k'}(\ell,\FF_p)$ and we will denote this subspace by $M_k(\ell,\FF_p)$ as well. 

As $k > 2$, $k' > p+1$, by \cite[Lemma $1.9$]{J} we know that there exists an integer $n >0$ such that $T_p^n(M_{k'}(\ell,\FF_p)) \subset M_k(\ell,\FF_p)$.
Hence, after localizing at $\fm$, we get that $T_p^n(M_{k'}(\ell,\FF_p)_{\fm}) \subset M_k(\ell,\FF_p)_{\fm}$. As $T_p -1 -p^{k'-1} \in \fm$, it follows that $T_p \in (\TT'_{\fm})^{\times}$. 
Therefore, $T_p$ is an invertible operator on $M_{k'}(\ell,\FF_p)_{\fm}$. So, it follows that $M_k(\ell,\FF_p)_{\fm} = M_{k'}(\ell,\FF_p)_{\fm}$.

As $p-1 \mid (k'-k)$ and $p-1 \nmid k$, it follows that $p-1 \nmid k'$.
Hence, by Lemma~\ref{quotient}, the largest quotient of $\TT'$ (resp. of $\TT$) acting faithfully on $M_{k'}(\ell,\FF_p)_{\fm}$ (resp. on $M_{k}(\ell,\FF_p)_{\fm}$) is $\TT'_{\fm}/(p)$ (resp. $\TT_{\fm}/(p)$).
Since $M_k(\ell,\FF_p)_{\fm} = M_{k'}(\ell,\FF_p)_{\fm}$, $\TT'_{\fm}/(p) \simeq \TT_{\fm}/(p)$.

Now suppose $I^{\eis}$ is principal.
Then $\TT_{\fm}$ is a quotient of $\ZZ_p \llbracket X \rrbracket$ which means $\dim(\tan(\TT_{\fm}/(p))) = 1$.
By combining Theorem~\ref{thm2} and Lemma~\ref{principal}, we get that the Eisenstein ideal of $\TT'_{\fm}$ is not principal.
Therefore, Lemma~\ref{principallem} implies that $\dim(\tan(\TT'_{\fm}/(p))) > 1$ and hence, $\dim(\tan(\TT_{\fm}/(p))) > 1$.
This gives us a contradiction.
Thus we conclude that $I^{\eis}$ is not principal.
So by Lemma~\ref{principal}, we get that $I^{\eis,0}$ is not principal. This finishes the proof of the theorem.
\end{proof}

We will now prove Corollaries~\ref{main}, \ref{corb} and \ref{corc}.
We begin with the proof of Corollary~\ref{main}.
\begin{proof}[Proof of Corollary~\ref{main}]
By combining Theorem~\ref{thm2} and Theorem~\ref{thm3}, we get that if $I^{\eis,0}$ is principal then $c_0 \cup b_0 \neq 0$.
Now suppose $c_0 \cup b_0 \neq 0$.
Then, by combining Lemma~\ref{tanlem}, Lemma~\ref{cuplem} and Lemma~\ref{pseudoordlem}, we get that $\dim(\tan(\TT_{\fm}/(p)))=1$.
So, Lemma~\ref{principallem} implies that $I^{\eis}$ is principal.
Hence, by Lemma~\ref{principal}, we get that $I^{\eis,0}$ is principal.
If Vandiver's conjecture holds for $p$, then we know, from Lemma~\ref{vanishlem}, that $c_0 \cup b_0 \neq 0$ if and only if $\prod_{i=1}^{p-1}(1-\zeta_p^i)^{i^{k-2}} \in (\ZZ/\ell\ZZ)^{\times}$ is not a $p$-th power.
This proves the corollary.
\end{proof}

Before proving Corollary~\ref{corb} and Corollary~\ref{corc}, we first prove a result that relates vanishing of cup product with class groups.

Recall that we denoted $ \QQ(\zeta_{\ell}^{(p)},\zeta_p)$ by $K$ and denoted its class group of by $\text{Cl}(K)$. 
Let $L$ be the unramified abelian extension of $K$ such that $\text{Gal}(L/K) = \text{Cl}(K)/\text{Cl}(K)^p$. Note that $L$ is also Galois over $\QQ$ and $\text{Gal}(L/K)$ is a normal subgroup of $\text{Gal}(L/\QQ)$.
As $\text{Gal}(L/K) = \text{Cl}(K)/\text{Cl}(K)^p$ is abelian, we get an action of $\text{Gal}(K/\QQ)$ on it. Now $\text{Gal}(K/\QQ) = \text{Gal}(\QQ(\zeta_{\ell}^{(p)})/\QQ) \times \text{Gal}(\QQ(\zeta_p)/\QQ)$. 
So we have an action of $\text{Gal}(\QQ(\zeta_p)/\QQ)$ on $\text{Cl}(K)/\text{Cl}(K)^p$.
Denote by $(\text{Cl}(K)/\text{Cl}(K)^p)[\omega_p^{1-k}]$ the subspace of $\text{Cl}(K)/\text{Cl}(K)^p$ on which $\text{Gal}(\QQ(\zeta_p)/\QQ)$ acts by $\omega_p^{1-k}$.

\begin{prop}
\label{classgroupprop}
Suppose $k$ is an even integer, $p-1 \nmid k$, and the $\omega_p^{1-k}$-component of the $p$-part of the class group of $\QQ(\zeta_p)$ is trivial.
The following are equivalent:
\begin{enumerate}
\item\label{part1} $c_0 \cup a_0 =0$,
\item\label{part2} $\dim((\text{Cl}(K)/\text{Cl}(K)^p)[\omega_p^{1-k}]) \geq 2$,
\item\label{part3} The image of $\prod_{i=1}^{\ell -1}i^{(\sum_{j=1}^{i-1}j^{k-1})}$ in $(\ZZ/\ell \ZZ)^{\times}$ is a $p$-th power.
\end{enumerate}
\end{prop}
\begin{proof}
As we are assuming that the $\omega_p^{1-k}$-component of the $p$-part of the class group of $\QQ(\zeta_p)$ is trivial, the equivalence between parts \eqref{part2} and \eqref{part3} follows from \cite[Theorem $1.9$]{L}.

To prove that part~\eqref{part1} implies part~\eqref{part2}, we follow the proof of \cite[Proposition $11.1.1$]{WWE1}.
Suppose $c_0 \cup a_0=0$.
Therefore, there exists a representation $\rho : G_{\QQ,p\ell} \to \GL_3(\FF_p)$ such that $$\rho(g) = \begin{pmatrix} \omega_p^{1-k}(g) & c_0(g) & F(g) \\ 0 & 1 & a_0(g) \\ 0 & 0 & 1\end{pmatrix} \text{ for all } g \in G_{\QQ,p\ell}.$$ Here $F : G_{\QQ,p\ell} \to \FF_p$ is a cochain such that the coboundary of $-F$ is $c_0 \cup a_0$.
From the proof of Lemma~\ref{inflem}, it follows that we can change $F$ by a suitable element of $H^1(G_{\QQ,p},\omega_p^{1-k})$ to assume $F(G_{\QQ_p})=0$.

Let $M$ be the extension of $\QQ$ fixed by $\ker(\rho)$. So $\text{Gal}(M/K) \simeq \ZZ/p\ZZ \times \ZZ/p\ZZ$ and its image under $\rho$ is $ \{\begin{pmatrix} 1 & a & b \\ 0 & 1 & 0 \\ 0 & 0 & 1\end{pmatrix} \mid a, b \in \FF_p\}$.
As $c_0$ is unramified at $p$ and $F(I_p)=0$, it follows that $M$ is unramified over all primes of $K$ lying above $p$.

As $\omega_p^{1-k}(I_{\ell}) =1$, it follows that $\rho(I_{\ell})$ is a $p$-group and hence, $\rho$ is tamely ramified at $\ell$. This means that $|\rho(I_{\ell})|=p$. So the image of $I_{\ell}$ in $\text{Gal}(M/\QQ)$ has cardinality $p$ and the image of $I_{\ell}$ in $\text{Gal}(K/\QQ)$ also has cardinality $p$. Hence, $M$ is unramified over all primes of $K$ lying above $\ell$.
Therefore, we conclude that $M$ is an unramified extension of $K$.

From the description of $\rho$ and the description of $\rho(\text{Gal}(M/K))$, it follows that $\text{Gal}(\QQ(\zeta_p)/\QQ)$ acts via $\omega_p^{1-k}$ on $\text{Gal}(M/K)$. 
As $\text{Gal}(M/K) \simeq \ZZ/p\ZZ \times \ZZ/p\ZZ$, it follows that $$\dim((\text{Cl}(K)/\text{Cl}(K)^p)[\omega_p^{1-k}]) \geq 2$$ which shows that part~\eqref{part1} implies part~\eqref{part2}.

We will now prove that part~\eqref{part2} implies part~\eqref{part1}.
Suppose $\dim((\text{Cl}(K)/\text{Cl}(K)^p)[\omega_p^{1-k}]) \geq 2$.
Let $M$ be the subfield of $L$ such that $\text{Gal}(M/K) = (\text{Cl}(K)/\text{Cl}(K)^p)[\omega_p^{1-k}]$. Note that $M$ is also Galois over $\QQ$.

So $V := \text{Gal}(M/K)$ is an $\FF_p$ vector space on which the cyclic $p$-group $\text{Gal}(\QQ(\zeta_{\ell}^{(p)})/\QQ)$ acts.
Let $\alpha$ be a generator of $\text{Gal}(\QQ(\zeta_{\ell}^{(p)})/\QQ)$.
Let $M'$ be the subfield of $M$ such that $\gal(M'/K) \simeq V/(\alpha-1)V$.
As $(\alpha-1)V$ is a subspace of $V$ stable under the action of $\gal(K/\QQ)$, it follows that $(\alpha-1)V$ is a normal subgroup of $\gal(M/\QQ)$ and hence, $M'$ is also Galois over $\QQ$.
As $\gal(\QQ(\zeta_{\ell}^{(p)})/\QQ)$ acts trivially on $\gal(M'/K)$, it follows that $\gal(M'/\QQ(\zeta_p))$ is an abelian $p$-group.
Note that $\gal(\QQ(\zeta_p)/\QQ)$ acts on $\gal(M'/K)$ via $\omega_p^{1-k}$ and it acts trivially on $\gal(K/\QQ(\zeta_p))$.
Hence, $\gal(M'/\QQ(\zeta_p)) \simeq \FF_p^{\oplus r}$ for some $r \geq 2$ and as a $\gal(\QQ(\zeta_p)/\QQ)$-representation, $\gal(M'/\QQ(\zeta_p)) \simeq \FF_p(\omega_p^{1-k})^{\oplus r-1} \oplus \FF_p$.

Therefore, from subfields of $M'$, we get $r-1$ elements of $H^1(G_{\QQ,p\ell},\omega_p^{1-k})$ which are linearly independent over $\FF_p$.
Now the prime of $\QQ(\zeta_p)$ lying above $p$ is unramified in $K$. As $M'$ is unramified over $K$, it follows that the prime of $\QQ(\zeta_p)$ lying above $p$ is also unramified in $M'$.
So the $r-1$ elements of $H^1(G_{\QQ,p\ell},\omega_p^{1-k})$ arising from subfields of $M'$ are all unramified at $p$. Hence, Lemma~\ref{cohomlem} implies that $r-1=1$. 
Therefore, $M'$ is a $\ZZ/p\ZZ$ extension of $K$.
Since $c_0$ generates the space of classes of $H^1(G_{\QQ,p\ell},\omega_p^{k-1})$ which are unramified at $p$, we get that $M'=K.K_{\rhob_{c_0}}$, where $K_{\rhob_{c_0}}$ is the extension of $\QQ$ fixed by $\ker(\rhob_{c_0})$.

Now let $M''$ be the subfield of $M$ such that $\gal(M''/K) \simeq V/(\alpha-1)^2V$. By our assumption, we have $\dim(V) \geq 2$ and we have just proved that $\dim(V/(\alpha-1)V) = 1$. So, we have $\dim(V/(\alpha-1)^2V) \geq 2$ and $\gal(M''/\QQ(\zeta_p))$ is not abelian as $\alpha$ does not act trivially on $V/(\alpha-1)^2V$.
Note that $K.K_{\rhob_{c_0}}=M' \subset M''$.
Denote the image of $G_{\QQ_{\ell}}$ in $\gal(M''/\QQ)$ by $D_{\ell}$.
As $\ell$ splits completely in $\QQ(\zeta_p)$, it follows that $D_{\ell}$ lies in $\gal(M''/\QQ(\zeta_p))$. 
Since $M''$ is unramified over $K$ and $K$ is abelian over $\QQ$, it follows that $D_{\ell}$ is abelian.

Suppose $c_0|_{G_{\QQ_{\ell}}} \neq \beta a_0|_{G_{\QQ_{\ell}}}$ for any $\beta \in \FF_p$. Note that $c_0|_{G_{\QQ_{\ell}}} \neq 0$ and $a_0|_{G_{\QQ_{\ell}}} \neq 0$. 
So the image of $D_{\ell}$ in $\gal(M'/\QQ(\zeta_p))$ has cardinality $p^2$ and hence, this image is all of $\gal(M'/\QQ(\zeta_p))$.
Now $\gal(M''/K)$ is an abelian group and it is normal in $\gal(M''/\QQ(\zeta_p))$ with their quotient given by $\gal(K/\QQ(\zeta_p)) \simeq \ZZ/p\ZZ$.
As $$\gal(M''/M') = (\alpha-1)V/(\alpha-1)^2V \subset \gal(M''/K),$$ $\gal(K/\QQ(\zeta_p))$ acts trivially on it. 
Hence, $\gal(M''/M')$ is in the center of $\gal(M''/\QQ(\zeta_p))$.

We have already seen that $D_{\ell}$ gets mapped surjectively on $\gal(M'/\QQ(\zeta_p))$ under the surjective map $\gal(M''/\QQ(\zeta_p)) \to \gal(M'/\QQ(\zeta_p))$.
Therefore, $\gal(M''/\QQ(\zeta_p)) = D_{\ell}.\gal(M''/M')$. As $D_{\ell}$ is abelian and $\gal(M''/M')$ is in the center of $\gal(M''/\QQ(\zeta_p))$, it follows that $\gal(M''/\QQ(\zeta_p))$ is abelian. But this gives a contradiction as we have already seen that $\gal(M''/\QQ(\zeta_p))$ is not abelian.

So we have $c_0|_{G_{\QQ_{\ell}}} = \beta a_0|_{G_{\QQ_{\ell}}}$ for some $\beta \in \FF_p$. This means that $c_0|_{G_{\QQ_{\ell}}} \cup a_0|_{G_{\QQ_{\ell}}} = 0$.
From \cite[Proposition $2.4.1$]{SS}, we get that $c_0 \cup a_0 =0$ which completes the proof of the proposition.
\end{proof}

\begin{proof}[Proof of Corollary~\ref{corb} and Corollary~\ref{corc}]
Corollary~\ref{corb} follows directly by Corollary~\ref{cor1} and Proposition~\ref{classgroupprop}.
Corollary~\ref{corc} follows directly by Theorem~\ref{thm3}, Proposition~\ref{classgroupprop} and Lemma~\ref{vanishlem}.
\end{proof}

We will end this article by proving the $R=\TT$ theorems mentioned in the introduction (Theorem~\ref{thmb}).
\begin{proof}[Proof of Theorem~\ref{thmb}]
{\bf Part~\eqref{item1} :}
Recall that we denoted by $(T,D) : G_{\QQ,p\ell} \to  R^{\ps,\ord}_{\rhob_0,k}(\ell)$ the universal pseudo-representation deforming $(\tr(\rhob_0),\det(\rhob_0))$.
Let $A=\begin{pmatrix} R^{\ps,\ord}_{\rhob_0,k}(\ell) & B \\ C & R^{\ps,\ord}_{\rhob_0,k}(\ell)\end{pmatrix}$ be the faithful GMA over $R^{\ps,\ord}_{\rhob_0,k}(\ell)$ and $\rho : G_{\QQ,p\ell} \to A^{\times}$ be the representation attached to $(T,D)$ by Lemma~\ref{gmalem}.
Suppose $\rho(i_{\ell}) =\begin{pmatrix} 1+x & b_{\ell} \\ c_{\ell} & 1-x \end{pmatrix}$ and $\rho(g_0) = \begin{pmatrix} a_0 & 0\\ 0 & d_0 \end{pmatrix}$.
Let $g_{\ell}$ be a lift of $\frob_{\ell}$ in $G_{\QQ_{\ell}}$ and suppose $\rho(g_{\ell}) =\begin{pmatrix} a & b \\ c & d \end{pmatrix}$.

Now Lemma~\ref{cuplem} implies that $\dim(\tan(R^{\ps,\ord}_{\rhob_0,k}(\ell)/(p)))=1$. So the proof of Lemma~\ref{tanlem} implies that any $p$-ordinary, $\ell$-unipotent pseudo-representation $(t,d) : G_{\QQ,p\ell} \to \FF_p[\epsilon]/(\epsilon^2)$ with determinant $\omega_p^{k-1}$ which deforms $(\tr(\rhob_0),\det(\rhob_0))$ is reducible.
Thus Lemma~\ref{tangenlem} implies that $m = (p,x)$ where $m$ is the maximal ideal of $R^{\ps,\ord}_{\rhob_0,k}(\ell)$.
As $\tr(\rho(g_0i_{\ell}))-\tr(\rho(g_0)) = (a_0-d_0)x \in J_0$ and $a_0-d_0 \in (R^{\ps,\ord}_{\rhob_0,k}(\ell))^{\times}$, it follows that $x \in J_0$. Since $m=(p,x)$ and $R^{\ps,\ord}_{\rhob_0,k}(\ell)/J_0 \simeq \ZZ_p$, we get that $J_0=(x)$.

Note that the relation $\rho(g_{\ell}i_{\ell}g_{\ell})^{-1} = \rho(i_{\ell})^{\ell}$ implies that $ab_{\ell} + b(1-x) = b(1+\ell x) + \ell b_{\ell} d$ (see proof of Theorem~\ref{thm2} for more details).
So we have $b_{\ell}(a-\ell d) = bx(1+\ell)$ which means $b_{\ell}c_{\ell}(a-\ell d) = bc_{\ell}x(1+\ell)$. 
By part~\eqref{bhaag4} of Lemma~\ref{gmalem}, we have $B=R^{\ps,\ord}_{\rhob_0,k}(\ell)b_{\ell}$. Thus $b = rb_{\ell}$ for some $r \in R^{\ps,\ord}_{\rhob_0,k}(\ell)$.
As $b_{\ell}c_{\ell}=-x^2$, we get that $x^2(a - \ell d) = x^3r(1+\ell)$.
Since $J_0 = (x)$ is the Eisenstein ideal of $R^{\ps,\ord}_{\rhob_0,k}(\ell)$, Lemma~\ref{redlem} implies that $a \equiv 1 \pmod{(x)}$ and $d \equiv \chi_p^{k-1}(\frob_{\ell}) \pmod{(x)}$.
Therefore, we get $x^2(1-\ell^k + xr')=x^3r(1+\ell)$ which means $x^2(1 - \ell^k + xr'')=0$.

Denote the image of $r \in R^{\ps,\ord}_{\rhob_0,k}(\ell)$ in $(R^{\ps,\ord}_{\rhob_0,k}(\ell))^{\red}$ by $\bar r$. As $J_0$ is a prime ideal of $R^{\ps,\ord}_{\rhob_0,k}(\ell)$, it contains the nilradical of $R^{\ps,\ord}_{\rhob_0,k}(\ell)$.
 Let $J^{\red}_0$ be the image of $J_0$ in $(R^{\ps,\ord}_{\rhob_0,k}(\ell))^{\red}$. So $J^{\red}_0=(\bar x)$. From previous paragraph, we get that $\bar{x}^2(1 - \ell^k + \overline{xr''})=0$.
Therefore, $\bar{x}(1 - \ell^k + \overline{xr''})=0$. As $1-\ell^k = p^{\nu+v_p(k)}.u$ for some $u \in \ZZ_p^{\times}$.
Hence, $|J^{\red}_0/(J^{\red}_0)^2| \leq p^{\nu+v_p(k)}$.

As $\TT_{\fm}$ is reduced, the map $\phi_{\TT} : R^{\ps,\ord}_{\rhob_0,k}(\ell) \to \TT_{\fm}$ factors through $(R^{\ps,\ord}_{\rhob_0,k}(\ell))^{\red}$ to get a map $\phi'_{\TT} : (R^{\ps,\ord}_{\rhob_0,k}(\ell))^{\red} \to \TT_{\fm}$.
Note that $\phi_{\TT}$ is a surjective morphism of augmented $\ZZ_p$-algebras and hence, $\phi'_{\TT}$ is also a surjective morphism of augmented $\ZZ_p$-algebras.
The kernels of the surjective morphisms $(R^{\ps,\ord}_{\rhob_0,k}(\ell))^{\red} \to \ZZ_p$ and $\TT_{\fm} \to \ZZ_p$ are $J^{\red}_0$ and $I^{\eis}$, respectively.

From Lemma~\ref{annih}, it follows that the annihilator of $I^{\eis}$ is $\ker(F)$.
Therefore, $\TT_{\fm}/(I^{\eis} +\ker(F)) \simeq \TT^0_{\fm}/I^{\eis,0}$ and \cite[Theorem $5.1.2$]{W} and \cite[Proposition II.$9.6$]{M2} imply that $\TT^0_{\fm}/I^{\eis,0} \simeq \ZZ/p^{\nu+v_p(k)}\ZZ$.
So, $|J^{\red}_0/(J^{\red}_0)^2| \leq |\TT^0_{\fm}/I^{\eis,0}|$. Hence, Wiles--Lenstra numerical criterion (\cite[Criterion I]{dS}) implies that $\phi'_{\TT}$ is an isomorphism of local complete intersection rings. This proves the first part of the theorem.

{\bf Part~\eqref{item2} :}
We now move on to the second part of the theorem.
Let $(T',D') : G_{\QQ,p\ell} \to  R^{\ps,\st}_{\rhob_0,k}(\ell)$ be the universal pseudo-representation deforming $(\tr(\rhob_0),\det(\rhob_0))$.
Let $A=\begin{pmatrix} R^{\ps,\st}_{\rhob_0,k}(\ell) & B \\ C & R^{\ps,\st}_{\rhob_0,k}(\ell)\end{pmatrix}$ be the faithful GMA over $R^{\ps,\st}_{\rhob_0,k}(\ell)$ and $\rho : G_{\QQ,p\ell} \to A^{\times}$ be the representation attached to $(T,D)$ by Lemma~\ref{gmalem}.
Suppose $\rho(i_{\ell}) =\begin{pmatrix} 1+x & b_{\ell} \\ c_{\ell} & 1-x \end{pmatrix}$.
Let $g_{\ell}$ be a lift of $\frob_{\ell}$ in $G_{\QQ_{\ell}}$ and suppose $\rho(g_{\ell}) =\begin{pmatrix} a & b \\ c & d \end{pmatrix}$.

Note that the pseudo-representation $(1+\chi_p^{k-1},\chi_p^{k-1}) : G_{\QQ,p\ell} \to \ZZ_p$ is unramified at $\ell$. Hence, $J_0$ contains the kernel of the surjective map $R^{\ps,\ord}_{\rhob_0,k}(\ell) \to R^{\ps,\st}_{\rhob_0,k}(\ell)$.
Let $J'_0$ be the image of $J_0$ in $R^{\ps,\st}_{\rhob_0,k}(\ell)$. Using the arguments of the first case of the theorem, we get that $J'_0 = (x)$.

Now the Steinberg-or-unramified at $\ell$ condition implies $\tr(\rho(g)(\rho(g_{\ell})-\ell^{k/2})(\rho(i_{\ell})-1))=0$ for all $g \in G_{\QQ,p\ell}$.
As $R^{\ps,\st}_{\rhob_0,k}(\ell)[\rho(G_{\QQ,p\ell})] = A$, it follows that $\tr(g'(\rho(g_{\ell})-\ell^{k/2})(\rho(i_{\ell})-1))=0$ for all $g' \in A$.
Putting $g' = \begin{pmatrix} 1 & 0 \\ 0 & 0 \end{pmatrix}$, we get $(a-\ell^{k/2})x+bc_{\ell}=0$.

As $B$ is generated by $b_{\ell}$ (by part~\eqref{bhaag4} of Lemma~\ref{gmalem}) and $b_{\ell}c_{\ell}=-x^2$, it follows that $bc_{\ell}=x^2r$ for some $r \in R^{\ps,\st}_{\rhob_0,k}(\ell)$.
As $J'_0=(x)$, Lemma~\ref{redlem} implies that $a \equiv 1 \pmod{(x)}$.
Therefore, $(a-\ell^{k/2})x+bc_{\ell}= x(1-\ell^{k/2} - xr' + xr) = 0$.
As $v_p(1-\ell^{k/2}) = \nu+v_p(k/2) = \nu + v_p(k)$, it follows that $|J'_0/(J'_0)^2| \leq p^{\nu+v_p(k)}$.

Note that $\psi_{\TT}$ is a surjective morphism of augmented $\ZZ_p$-algebras.
The kernels of the surjective morphisms $R^{\ps,\st}_{\rhob_0,k}(\ell) \to \ZZ_p$ and $\TT_{\fm} \to \ZZ_p$ are $J'_0$ and $I^{\eis}$, respectively.
We have already seen in the proof of part $1$ of the theorem that $\TT_{\fm}/(I^{\eis} +\ker(F)) \simeq \TT^0_{\fm}/I^{\eis,0} \simeq \ZZ/p^{\nu+v_p(k)}\ZZ$. Hence, we have $|J'_0/(J'_0)^2| \leq |\TT^0_{\fm}/I^{\eis,0}|$.
So Wiles--Lenstra numerical criterion (\cite[Criterion I]{dS}) implies that $\psi_{\TT}$ is an isomorphism of local complete intersection rings. This proves the second part of the theorem.

{\bf Part~\eqref{item3} :}
We now come to the last part of the theorem.
By \cite[Lemma $3.1$]{D3}, one can choose the universal deformation $\rho^{\univ,\ell} : G_{\QQ,p\ell} \to \GL_2(R^{\defo,\ord}_{\rhob_{c_0},k}(\ell))$ of $\rhob_{c_0}$ such that $\rho^{\univ,\ell}(g_0) = \begin{pmatrix} a_0 & 0 \\ 0 & d_0 \end{pmatrix}$. 

Suppose $\rho^{\univ,\ell}(i_{\ell}) = \begin{pmatrix} 1+x & b_{\ell} \\ c_{\ell} & 1-x\end{pmatrix}$. 
Let $\phi : R^{\ps,\ord}_{\rhob_0,k}(\ell) \to R^{\defo,\ord}_{\rhob_{c_0},k}(\ell)$ be the surjective map induced by $(\tr(\rho^{\univ,\ell}),\det(\rho^{\univ,\ell}))$.
Now we know, from the proof of the first part of the theorem, that $J_0 = (T(g_0i_{\ell})-T(g_0))$. The ideal generated by $\phi(J_0)$ is $I_0$.
So $I_0=(\tr(\rho^{\univ,\ell}(g_0i_{\ell})) - \tr(\rho^{\univ,\ell}(g_0))) = (x)$.
As the ideal generated by $\phi_{\TT^0}(I_0)$ is the Eisenstein ideal $I^{\eis,0}$ of $\TT^0_{\fm}$.
So $I^{\eis,0} = (\phi_{\TT^0}(x))$.

As $I_0=(x)$, it follows that the maximal ideal of $R^{\defo,\ord}_{\rhob_{c_0},k}(\ell)$ is generated by $p$ and $x$.
So we have a surjective map $F_1 : \ZZ_p \llbracket X \rrbracket \to R^{\defo,\ord}_{\rhob_{c_0},k}(\ell)$ which sends $X$ to $x$.
Composing $F_1$ with $\phi_{\TT^0}$, we get a surjective map $F_2 : \ZZ_p \llbracket X \rrbracket \to \TT^0_{\fm}$ such that $F_2(X) = \phi_{\TT^0}(x)$.
As $\ZZ_p \llbracket X \rrbracket$ is a UFD and $\TT^0_{\fm}$ is finite and flat over $\ZZ_p$, it follows that $\ker(F_2)$ is a principal ideal.
Now \cite[Theorem $5.1.2$]{W} and \cite[Proposition II.$9.6$]{M2} imply that $\TT^0_{\fm}/(\phi_{\TT^0}(x)) = \TT^0_{\fm}/I^{\eis,0} \simeq \ZZ/p^{\nu+v_p(k)}\ZZ$.
Hence, we can find a generator $\alpha$ of $\ker(F_2)$ such that $\alpha = p^{\nu+v_p(k)}+Xf(X)$ for some $f(X) \in \ZZ_p \llbracket X \rrbracket$.

Since $I_0=(x) = (F_1(X))$, Lemma~\ref{eisenlem} implies that there exists a $\beta \in \ker(F_1)$ such that $\beta =p^e + Xg(X)$ for some $g(X) \in \ZZ_p \llbracket X \rrbracket$ and $e \leq \nu+v_p(k)$.
As $\ker(F_1) \subset \ker(F_2)$, it follows that there exists some $h(X) \in \ZZ_p \llbracket X \rrbracket$ such that $p^e + Xg(X) = h(X)(p^{\nu+v_p(k)}+Xf(X))$.
So the constant term of $h(x)$ is $p^{e'}$ for some $e' \geq 0$. Now $e \leq \nu+v_p(k)$ which means $e'=0$. Hence, $h(X)$ is a unit which means $\ker(F_2) \subset \ker(F_1)$.
Therefore, we conclude that $\ker(F_2) = \ker(F_1)$. As $F_1$ is surjective, it follows that $\phi_{\TT^0}$ is injective. As $\phi_{\TT^0}$ is also surjective, it follows that $\phi_{\TT^0}$ is an isomorphsim which proves the final part of the theorem.
\end{proof}



\begin{thebibliography}{9}

\bibitem{B1} J. Bella\"{i}che,
\emph{Pseudodeformations},
Math. Z. 270 (2012), no. 3-4, 1163-1180.

\bibitem{Bel} J. Bella\"{i}che,
\emph{Image of pseudo-representations and coefficients of modular forms modulo $p$,}
Adv. Math., 353 (2019), 647-721.

\bibitem{BC1} J. Bella\"{i}che and G. Chenevier,
\emph{Lissite de la courbe de Hecke de $\GL(2)$ aux points Eisenstein critiques,}
 J. Inst. Math. Jussieu (2006) 5(2), 333-349. 

\bibitem{BC} J. Bella\"{i}che and G. Chenevier,
\emph{Families of Galois representations and Selmer groups,}
Asterisque 324 (2009), pp. xii+314.                                             

\bibitem{BK} J. Bella\"{i}che and C. Khare,
\emph{Level $1$ Hecke algebras of modular forms modulo $p$,}
Compos. Math. 151. 3 (2015), 397-415.

\bibitem{CE} F. Calegari and M. Emerton,
\emph{On the ramification of Hecke algebras at Eisenstein primes,}
 Invent. Math. 160 (2005), no. 1, 97–144. 

\bibitem{CS} F. Calegari and J. Specter,
\emph{Pseudorepresentations of weight one are unramified,}
 Algebra Number Theory 13 (2019), no. 7, 1583-1596. 

\bibitem{C} G. Chenevier,
\emph{The $p$-adic analytic space of pseudocharacters of a profinite group and pseudorepresentations over arbitrary rings,}
Automorphic forms and Galois representations. Vol. 1, 221-285, London Math. Soc. Lecture Note Ser., 414, Cambridge Univ. Press, Cambridge, 2014.


\bibitem{D2} S. V. Deo,
\emph{Effect of increasing the ramification on pseudo-deformation rings,}
J. Th\'{e}orie des Nombres Bordeaux 34 (2022), no. 1, 189--236.

\bibitem{D3} S. V. Deo,
\emph{On density of modular points in pseudo-deformation rings,}
Preprint, Available at \href{https://arxiv.org/abs/2105.05823.pdf}{https://arxiv.org/abs/2105.05823.pdf}.

\bibitem{dS} B. de Smit, K. Rubin and R. Schoof ,
\emph{Criteria for complete intersections,}
Modular forms and Fermat's last theorem (Boston, MA, 1995), 343-356, Springer, New York, 1997.

\bibitem{E} D. Eisenbud,
\emph{Commutative algebra with a view toward algebraic geometry,}
Graduate Texts in Mathematics, 150, Springer-Verlag, New York, 1995.

\bibitem{G} B. Gross,
\emph{A tameness criterion for Galois representations associated to modular forms $\pmod{p}$,}
Duke Math. J. 61 (1990), 445--517. 

\bibitem{J} N. Jochnowitz,
\emph{Congruences between systems of eigenvalues of modular forms,}
 Trans. Amer. Math. Soc. 270 (1982), no. 1, 269-285. 

\bibitem{Kup} M. D. Sikiri\'{c}, P. Elbaz-Vincent, A. Kupers and J. Martinet,
\emph{Voronoi complexes in higher dimensions, cohomology of $\GL_N(\ZZ)$ for $N \geq 8$ and the triviality of $K_8(\ZZ)$,}
Preprint, Available at \href{https://arxiv.org/pdf/1910.11598.pdf}{https://arxiv.org/pdf/1910.11598.pdf}.

\bibitem{K} M. Kurihara,
\emph{Some remarks on conjectures about cyclotomic fields and $K$-groups of $\mathbb{Z}$,}
Compos. Math. 81 (1992), 223--236.

\bibitem{L} E. Lecouturier,
\emph{On the Galois structure of the class group of certain Kummer extensions,}
  J. Lond. Math. Soc. (2) 98 (2018), no. 1, 35–58.

\bibitem{L2} E. Lecouturier,
\emph{Higher Eisenstein elements, higher Eichler formulas and rank of Hecke algebras,}
 Invent. Math. 223 (2021), no. 2, 485–595.

\bibitem{M2} B. Mazur,
\emph{Modular curves and the Eisenstein ideal,}
 Inst. Hautes \'{E}tudes Sci. Publ. Math. No. 47 (1977), 33-186. 

\bibitem{M} B. Mazur,
\emph{Deforming Galois representations,}
Galois groups over $\mathbb{Q}$ (Berkeley, CA, 1987), 385-437, Math. Sci. Res. Inst. Publ., 16, Springer, New York, 1989.

 \bibitem{Mer} L. Merel,
\emph{L'accouplement de Weil entre le sous-groupe de Shimura et le sous-groupe cuspidal de $J_0(p)$,}
  J. Reine Angew. Math. 477 (1996), 71-115.

\bibitem{Oh} M. Ohta,
\emph{Eisenstein ideals and the rational torsion subgroups of modular Jacobian varieties II,}
Tokyo J. Math. 37 (2014), 273--318.

\bibitem{O} R. Osburn,
\emph{Vanishing of eigenspaces and cyclotomic fields,}
Int. Math. Res. Not. 20 (2005), 1195--1202.

\bibitem{Ra} R. Ramakrishna,
\emph{On a Variation of Mazur's deformation functor,}
Compos. Math. 87 (1993), 269-286.

\bibitem{SS} K. Schaefer and E. Stubley,
\emph{Class groups of Kummer extensions via cup product in Galois cohomology,}
 Trans. Amer. Math. Soc. 372 (2019), no. 10, 6927–6980. 

\bibitem{W} P. Wake,
\emph{The Eisenstein ideal for weight $k$ and a Bloch-Kato conjecture for tame families,}
To appear in J. Eur. Math. Soc.

\bibitem{WWE} P. Wake and C. Wang-Erickson,
\emph{Deformation conditions for pseudorepresentations,}
 Forum Math. Sigma 7 (2019), Paper No. e20, 44 pp. 

\bibitem{WWE1} P. Wake and C. Wang-Erickson,
\emph{The rank of Mazur's Eisenstein ideal,}
Duke Math. J. 169 (2020), no. 1, 31-115.

\bibitem{WWE2} P. Wake and C. Wang-Erickson,
\emph{The Eisenstein ideal with squarefree level,}
 Adv. Math. 380 (2021), 107543, 62 pp.

\bibitem{Wash} L. Washington,
\emph{Introduction to cyclotomic fields,}
Graduate Texts in Matheatics, 83, Springer-Verlag, New York, 1982.

\bibitem{Wa} L. Washington,
\emph{Galois Cohomology,}
Modular forms and Fermat's last theorem (Boston, MA, 1995), 101-120, Springer, New York, 1997.


\end{thebibliography}
\end{document}